\documentclass[12pt]{article}

\usepackage{amsfonts,amssymb,amsmath,exscale,relsize,enumitem,amsthm,mathtools, physics}
\usepackage{graphicx}
\usepackage{float}
\usepackage[thinlines]{easytable}
\usepackage{csquotes}
\usepackage{setspace}

\usepackage[authoryear,round]{natbib}
\usepackage[
colorlinks=true,
citecolor=blue,
linkcolor=blue,
urlcolor=blue,
pagebackref=true
]{hyperref}

\usepackage{varioref}
\usepackage{subcaption}
\usepackage{xcolor}
\usepackage{setspace}
\usepackage{appendix}
\usepackage[margin=1.35in]{geometry}
\usepackage{setspace} 

\newtheorem{theorem}{Theorem}
\newtheorem{lemma}{Lemma}
\newtheorem{corollary}{Corollary}
\newtheorem{assumption}{Assumption}
\newtheorem{remark}{Remark}

\title{Asymptotic Bayes Optimality Under Sparsity of the Gavrilov--Benjamini--Sarkar Step-Down Testing Procedure}
\author{
	Prasenjit Ghosh\thanks{Department of Statistics, Texas A\&M University, College Station, TX 77843, USA. Email: \texttt{prasenjit@stat.tamu.edu}}%
	\hspace{1em} 
	Arijit Chakrabarti\thanks{Applied Statistics Unit, Indian Statistical Institute, Kolkata - 700108, India. Email: \texttt{arc@isical.ac.in}}%
}

\date{} 

\begin{document}
	
	\maketitle
	
\begin{abstract}
	
In this article, we investigate the asymptotic Bayes optimality under sparsity (ABOS) of the Gavrilov--Benjamini--Sarkar (GBS) step-down multiple testing procedure \citep{GBS2009} in the sparse Gaussian sequence model. While the asymptotic optimality properties of the Benjamini--Hochberg procedure have been extensively studied, corresponding results for the GBS procedure remain unavailable despite its favorable finite-sample performance in several settings. Within the spike-and-slab Bayesian formulation and the asymptotic decision-theoretic framework of \citet{BCFG2011}, we establish that the GBS procedure is ABOS over a broad class of sparse asymptotic regimes. Existing analyses of the Benjamini--Hochberg procedure rely on approximating the random rejection threshold by a suitable deterministic surrogate. In contrast, our approach proceeds through a direct analysis of the Bayes risk by separately controlling the false discovery and false nondiscovery components. The analysis combines two key ingredients: a new finite-sample inequality for shifted order statistics associated with the GBS critical constants and a signal-crossing argument for ordered alternative \(p\)-values that exploits the sequential step-down structure of the procedure. For diverging expected signal counts, these tools yield ABOS without threshold localization. For bounded positive expected signal counts, a direct argument based on the first GBS critical constant establishes ABOS as well. In particular, our results cover polynomial sparsity $p_m\asymp m^{-\beta}$ for every fixed $0<\beta\le1$. To the best of our knowledge, this is the first asymptotic decision-theoretic analysis of the GBS procedure and the first proof of its asymptotic Bayes optimality under sparsity.

	\medskip
	
	
	\medskip
	
	
\end{abstract}

\section{Introduction}

Sparse signal discovery has emerged as one of the central challenges in modern statistical inference. Contemporary scientific investigations routinely involve the simultaneous examination of a very large number of candidate variables in order to identify a relatively small number of meaningful signals hidden among a vast collection of null effects. Such problems arise naturally in genomics, bioinformatics, neuroimaging, astronomy, finance, economics, medicine, environmental science, and numerous other disciplines. The resulting inferential task is inherently high-dimensional and requires statistical procedures capable of distinguishing true signals from noise while maintaining adequate control over erroneous decisions. Consequently, multiple hypothesis testing has become a fundamental tool in the analysis of large-scale datasets and has attracted sustained attention over the past several decades.

The need for simultaneous inference naturally introduces multiplicity issues, since procedures designed for individual testing problems may accumulate errors rapidly when applied repeatedly across a large collection of hypotheses. Consequently, a vast literature has developed on multiple testing procedures controlling global measures of Type-I error. Among the most widely studied criteria are the family-wise error rate (FWER), defined as the probability of making at least one false rejection, and the false discovery rate (FDR), defined as the expected proportion of false discoveries among all rejected hypotheses. While classical procedures such as the Bonferroni correction provide rigorous control of the FWER, they often become excessively conservative in large-scale problems involving thousands or even millions of simultaneous tests. This limitation has motivated the development of less stringent error criteria that allow for a more favorable balance between false discoveries and false non-discoveries.

The introduction of the false discovery rate by \citet{BH1995} marked a major turning point in the multiple testing literature. Under independence, \citet{BH1995} demonstrated that a simple step-up procedure controls the FDR at a prescribed level while often achieving substantially greater power than traditional FWER-controlling methods. Since then, FDR control has become one of the dominant paradigms in large-scale inference, leading to an extensive body of methodological developments; see, for example, \citet{BL1999}, \citet{BY2001}, \citet{STO_2002}, \citet{STS_2004}, \citet{BKY2006}, \citet{SKS2008}, \citet{BLROQ2009}, \citet{GBS2009}, \citet{SUN_CAI_2009}, and the references therein. Owing to their ability to maintain a favorable balance between false discoveries and false non-discoveries, FDR-based methodologies have become central to modern multiple testing theory and practice.

Alongside the development of procedures controlling global error rates, considerable effort has been devoted to developing and analyzing multiple testing procedures within the Bayesian paradigm. While FWER- and FDR-based methodologies focus primarily on controlling specific measures of Type-I error, Bayesian formulations provide a complementary framework in which multiple testing procedures can be evaluated through their overall decision-theoretic performance. A natural Bayesian formulation of sparse multiple testing problems is provided by the two-groups model, in which each unknown parameter is assumed to arise from a mixture of a null component and a signal component. Such formulations, often referred to as spike-and-slab models, have played a central role in the Bayesian analysis of sparse high-dimensional problems. The two-groups model provides a convenient probabilistic framework for distinguishing signals from noise while simultaneously allowing information to be borrowed across tests through the common model parameters. Fully Bayesian and empirical Bayes approaches based on this formulation have been widely studied in the multiple testing literature; see, for example, \citet{SB2006}, \citet{Efron2004}, \citet{Storey2007}, \citet{Efron2008}, and \citet{BGT2008}. Moreover, \citet{CPS2009} remarked that a suitably chosen two-groups formulation may be regarded as a \enquote{gold standard} for sparse inference problems.

Within the two-groups framework, each multiple testing procedure can be evaluated through its Bayes risk, typically defined under an additive loss function that assigns separate penalties to false discoveries and false non-discoveries. This formulation naturally leads to the Bayes Oracle, namely the decision rule that minimizes the Bayes risk under the underlying two-groups model. Since the Bayes Oracle depends on unknown model parameters and is therefore generally unavailable in practice, a fundamental question is whether practically implementable multiple testing procedures can asymptotically attain the Oracle risk. This question gave rise to the notion of asymptotic Bayes optimality under sparsity (ABOS), introduced by \citet{BCFG2011}. A testing procedure is said to possess the ABOS property if its Bayes risk is asymptotically equivalent to the Oracle risk, that is, if the ratio of the two risks converges to one as the number of tests grows to infinity under sparsity. Consequently, ABOS provides a natural decision-theoretic yardstick for evaluating large-scale multiple testing procedures and has emerged as one of the most influential criteria for assessing whether practically implementable methods can asymptotically emulate ideal Oracle performance.

A substantial body of work has subsequently developed around the study of optimal multiple testing procedures under sparsity from both Bayesian and frequentist perspectives. Important themes include asymptotic Bayes optimality under sparsity (ABOS), Oracle risk approximations, local false discovery rate methodologies, and minimaxity considerations; see, for example, \citet{ABDJ2006}, \citet{SUN_CAI_2007}, \citet{BGT2008}, \citet{BCFG2011}, \citet{NR2012}, \citet{DG2013}, \citet{GTGC2015}, \citet{GC2016}, \citet{PaulChakrabarti2025AISM}, and references therein. For frequentist procedures, \citet{BGT2008} and \citet{BCFG2011} established asymptotic Bayes optimality results for the Bonferroni correction and the Benjamini--Hochberg procedure under suitable sparse asymptotic regimes, thereby providing some of the earliest decision-theoretic justifications for widely used multiple testing methodologies. Related developments have also appeared in the local false discovery rate literature; see, for example, \citet{SUN_CAI_2007} and \citet{NR2012}.

More recently, considerable attention has been devoted to Bayesian multiple testing procedures induced by continuous shrinkage priors. A series of works has investigated their asymptotic properties from both estimation and testing perspectives; see, for example, \citet{DG2013}, \citet{GTGC2015}, \citet{GC2016}, \citet{Paul2025ABOSCount}, and \citet{PaulChakrabarti2025AISM}, among others. Together, these investigations established that testing procedures induced by one-group global--local shrinkage priors can attain asymptotic Bayes optimality under sparsity across a variety of sparse inference settings. More recently, attention has shifted toward increasingly refined characterizations of the fundamental limits of sparse multiple testing. Notable developments include the sharp multiple testing boundary established by \citet{ACR2024} and the sharp asymptotic minimaxity results for testing procedures induced by one-group global--local shrinkage priors obtained by \citet{PaulGhoshChakrabarti2025}. Collectively, these works have significantly advanced our understanding of Bayes-optimal sparse signal recovery under independence and clarified the extent to which computationally feasible procedures can approximate ideal Oracle performance.

Among the many procedures proposed for FDR control, the Benjamini--Hochberg (BH) procedure \citep{BH1995} occupies a particularly prominent position. Its simplicity, strong empirical performance, and favorable asymptotic properties have made it one of the most widely used multiple testing procedures in modern statistics. Motivated by the desire to increase power while retaining rigorous false discovery rate control, \citet{GBS2009} introduced the Gavrilov--Benjamini--Sarkar (GBS) step-down procedure. Unlike the BH procedure, which employs a step-up mechanism, the GBS procedure utilizes a step-down structure and was specifically designed to increase the number of discoveries while maintaining FDR control under independence. The simulation studies reported in \citet{GBS2009} exhibited settings in which the GBS procedure achieved greater power than the Benjamini--Hochberg procedure while maintaining false discovery rate control; such comparisons depend on the underlying configuration and do not imply uniform dominance. Consequently, the GBS procedure has emerged as one of the most prominent alternatives to the BH procedure in the independence setting and has attracted considerable attention in both the methodological and applied multiple testing literature. From a decision-theoretic perspective, it is therefore natural to ask whether the GBS procedure possesses asymptotically optimal risk behavior.

Despite its attractive finite-sample performance and widespread use, relatively little is known about the asymptotic decision-theoretic behavior of the GBS procedure. In particular, while asymptotic Bayes optimality under sparsity has been established for the BH procedure and several Bayesian multiple testing procedures, corresponding results for the GBS procedure appear to be unavailable. At first glance, one might expect such results to follow from existing analyses of the BH procedure. However, the arguments used in the BH literature rely heavily on the specific structure of the BH rejection rule and do not transfer directly to the GBS setting. Consequently, it remains unclear whether the favorable finite-sample properties of the GBS procedure are accompanied by asymptotically optimal risk behavior in sparse high-dimensional settings.


More recently, investigations of dependence-aware multiple testing procedures have provided additional evidence regarding the decision-theoretic performance of the GBS methodology. In particular, \citet{GhoshChakrabartiBSD2026} studied the MRD--GBS method \citep{GhoshChakrabartiMRDGBS2026}, a residual-based step-down testing procedure designed to extend the GBS philosophy to dependent settings. Under independence, MRD--GBS reduces exactly to the classical GBS procedure. Extensive simulation studies reported therein revealed a striking empirical phenomenon: the Bayes risk of MRD--GBS remained remarkably close to that of the Bayes Oracle across a wide range of sparse Gaussian one-factor models, with independence arising as a special case. This near-Oracle behavior persisted even in relatively small dimensions. Since MRD--GBS coincides with GBS under independence, these findings provide compelling empirical evidence that the GBS procedure itself may possess favorable Oracle-like risk properties under sparsity.

The primary objective of this paper is to establish this
property rigorously. Within the spike-and-slab Bayesian
formulation and the asymptotic decision-theoretic framework
of \citet{BCFG2011}, we investigate the asymptotic Bayes
optimality under sparsity (ABOS) of the
Gavrilov--Benjamini--Sarkar (GBS) step-down procedure
\citep{GBS2009} in the sparse Gaussian sequence model
\eqref{eq:gaussian_sequence_model}. A central ingredient
in existing ABOS analyses of the Benjamini--Hochberg
procedure is the approximation of its random rejection
threshold by a suitable deterministic surrogate, often
arising from Bayesian false discovery rate considerations;
see, for example, \citet{BCFG2011} and
\citet{FrommletBogdan2013}. The sequential nature of the
GBS rejection rule makes such threshold-localization
arguments considerably less transparent. We therefore
develop a direct Bayes-risk analysis that exploits the
structure of the GBS critical constants and its step-down
rejection mechanism, without requiring localization of
the random rejection threshold.

Our theoretical analysis proceeds by separately controlling
the false-discovery and false-nondiscovery contributions
to the Bayes risk. A key ingredient is a new finite-sample
inequality for shifted uniform order statistics associated
with the GBS critical constants, which yields explicit
control of the Type-I error contribution. For sparse
regimes in which the expected number of non-null
hypotheses diverges, we develop a signal-crossing
argument for ordered alternative $p$-values that exploits
the sequential structure of the step-down procedure
and provides asymptotic control of the Type-II error
contribution. We then establish a complementary ABOS
result for the substantially sparser regime in which
the expected number of non-null hypotheses converges
to a finite positive limit. In this regime, the
signal-crossing analysis is replaced by a direct
argument based on the first GBS critical constant,
which controls the expected number of missed signals
without requiring concentration of the signal count.

Together, these results establish the ABOS property
of the GBS procedure across both diverging and bounded
positive expected signal counts. In particular, under
suitable signal-strength and nominal FDR calibration
conditions, we establish ABOS throughout the polynomial
sparsity regime $p_m\asymp m^{-\beta}$ for every fixed
$0<\beta\leq 1$, including the boundary case
$\beta=1$. To the best of our knowledge, these results
provide the first rigorous ABOS theory for the GBS
step-down procedure. Beyond establishing the
asymptotic decision-theoretic optimality of GBS,
our analysis demonstrates how finite-sample
order-statistic inequalities and direct control
of missed signals can provide alternatives to
threshold-localization techniques in the study
of sequential multiple testing procedures.

The remainder of the paper is organized as follows.
In Section~\ref{sec:Model}, we introduce the sparse Gaussian
sequence model, the Bayes risk framework, the asymptotic
assumptions, and the Gavrilov--Benjamini--Sarkar procedure.
Section~\ref{sec:ABOS-GBS-Theory} develops the asymptotic
Bayes risk analysis of the GBS procedure. In particular,
Section~\ref{sec:GBS-TypeI} establishes finite-sample control
of the Type-I error component through an inequality for
shifted order statistics associated with the GBS critical
constants. Section~\ref{sec:GBS-TypeII} develops the Type-II
error analysis based on a signal-crossing argument for
ordered alternative $p$-values. The main ABOS results are
presented in Section~\ref{sec:ABOS-GBS-Main-Result},
covering both diverging and bounded positive expected
signal counts and yielding an extension throughout the
polynomial sparsity regime. Section~\ref{sec:Discussion}
concludes with a discussion of the results and directions
for future research. Proofs of the auxiliary lemmas and
the finite-sample Type-I error bound are provided in
the Appendix, while the main ABOS results are proved
in Section~\ref{sec:ABOS-GBS-Main-Result}.


\section{Model Formulation and Preliminaries}
\label{sec:Model}

\subsection{Sparse Gaussian Sequence Model}

In this article, we consider the problem of simultaneously testing the means of a collection of independent Gaussian random variables. Let \(X_1,\ldots,X_m\) be independent observations satisfying
\begin{equation}\label{eq:gaussian_sequence_model}
X_i \mid \mu_i \sim N(\mu_i,\sigma^2),
\qquad i=1,\ldots,m,
\end{equation}

where $\sigma^2>0$ is assumed known.

Our goal is to simultaneously test the hypotheses
\begin{equation}\label{eq:testing_problem}
H_{0i}:\mu_i=0
\qquad \text{versus} \qquad
H_{Ai}:\mu_i\neq 0,
\qquad i=1,\ldots,m.
\end{equation}

%

Following the asymptotic decision-theoretic framework of \citet{BCFG2011}, we assume that the unknown means \(\mu_1,\dots,\mu_m\) arise from the two-groups model
\begin{equation}
	\label{eq:two_groups_mu}
	\mu_i
	\stackrel{\mathrm{i.i.d.}}{\sim}
	(1-p_m)\delta_{\{0\}}
	+
	p_mN(0,\psi_m^2),
	\qquad i=1,\ldots,m,
\end{equation}
where $p_m\in(0,1)$ denotes the proportion of non-null effects and $\delta_{\{0\}}$ denotes the Delta measure degenerate at zero. Equivalently, the marginal distribution of $X_i$ is given by
\begin{equation}
	\label{eq:two_groups_x}
	X_i
	\stackrel{\mathrm{i.i.d.}}{\sim}
	(1-p_m)N(0,\sigma^2)
	+
	p_mN(0,\sigma^2+\psi_m^2),
	\qquad i=1,\ldots,m.
\end{equation}

Throughout the paper, we assume without loss of generality that $\sigma^2=1$. The parameter $p_m$ is allowed to depend on $m$ and is assumed to converge to zero as $m\to\infty$, reflecting the sparse nature of the underlying signal configuration. The variance parameter $\psi_m^2$ determines the strength of the non-null effects and is also allowed to vary with $m$.

Introducing latent indicators \(\nu_1,\ldots,\nu_m\), where
\[
\nu_i=
\begin{cases}
	0, & \text{if } H_{0i}\text{ is true},\\
	1, & \text{if } H_{1i}\text{ is true},
\end{cases}
\]
the above multiple testing problem \eqref{eq:testing_problem} may be equivalently formulated as
\begin{equation}\label{eq:testing_problem_latent}
	H_{0i}:\nu_i=0
	\qquad \text{versus} \qquad
	H_{1i}:\nu_i=1,
	\quad i=1,\ldots,m.
\end{equation}

We denote by
\begin{equation*}
	\mathcal H_0
	=
	\{1\leq i \leq m:\nu_i=0\}
	\qquad\text{and}\qquad
	\mathcal H_0^c
	=
	\{1\leq i \leq m:\nu_i=1\}
\end{equation*}

the sets of null and non-null hypotheses, respectively.

\subsection{Bayes Risk and the Bayes Oracle}

We evaluate multiple testing procedures under the additive loss framework of
\citet{BCFG2011}. Let $\delta_{0,m}>0$ and $\delta_{A,m}>0$ denote the losses
associated with a false discovery and a false nondiscovery, respectively. If
$V_m$ and $T_m$ denote the numbers of false discoveries and false
nondiscoveries produced by a multiple testing rule, then its Bayes risk is
defined as
\begin{equation}
	\label{eq:bayes_risk}
	R_m
	=
	\delta_{0,m}\mathbb E(V_m)
	+
	\delta_{A,m}\mathbb E(T_m),
\end{equation}
where the expectation is taken under the probability law induced by the
two-groups model.

Let
\[
f_m=\frac{1-p_m}{p_m}
\]
denote the sparsity parameter and define
\[
\delta_m=\frac{\delta_{0,m}}{\delta_{A,m}}
\]
to be the ratio between the losses associated with a false discovery and a
false nondiscovery. As shown by \citet{BGT2008} and
\citet{BCFG2011}, the Bayes Oracle, namely the multiple testing rule that
minimizes the Bayes risk \eqref{eq:bayes_risk}, rejects $H_{0i}$ whenever
\begin{equation}
	\label{eq:oracle_rule}
	\frac{f(X_i\mid\nu_i=1)}
	{f(X_i\mid\nu_i=0)}
	>
	\delta_mf_m.
\end{equation}

Under the Gaussian two-groups model \eqref{eq:two_groups_x} with
$\sigma^2=1$, straightforward calculations show that
\eqref{eq:oracle_rule} is equivalent to rejecting $H_{0i}$ whenever
\begin{equation}
	\label{eq:oracle_rule_equiv}
	X_i^2>c_m^2,
\end{equation}
where
\[
c_m^2
=
\left(1+\frac{1}{\psi_m^2}\right)
\left\{
\log(\psi_m^2\delta_m^2f_m^2)
+
\log\left(1+\frac{1}{\psi_m^2}\right)
\right\}.
\]

Following the asymptotic framework of \citet{BCFG2011}, we introduce
\[
u_m=\psi_m^2,
\qquad
v_m=u_mf_m^2\delta_m^2,
\]
so that the oracle threshold may be written as
\begin{equation}
	\label{eq:oracle_threshold}
c_m^2
=
\left(1+\frac{1}{u_m}\right)
\left\{
\log(v_m)
+
\log\left(1+\frac{1}{u_m}\right)
\right\}.
\end{equation}

Since the Bayes Oracle depends on the unknown sparsity parameter $p_m$, the
signal variance $\psi_m^2$, and the loss ratio $\delta_m$, it is unattainable
in practice. Nevertheless, its Bayes risk provides the natural benchmark
against which the asymptotic performance of practical multiple testing
procedures is assessed throughout this paper.

\subsection{Sparse Asymptotic Framework}
For the main signal-crossing analysis, we use the following sparse asymptotic framework, adapted from \citet{BCFG2011}. A separate bounded-signal regime is considered in Section~\ref{sec:ABOS-GBS-Main-Result}.

\begin{assumption}
	\label{ass:BCFG}
	The sequence \((p_m,\psi_m^2,\delta_m)\) satisfies
	\begin{itemize}
		\item[(i)] \(p_m\to0\) as \(m\to\infty\);
		\item[(ii)] \(mp_m\to\infty\) as \(m\to\infty\);
		\item[(iii)] \(u_m=\psi_m^2\to\infty\) as \(m\to\infty\);
		\item[(iv)] \(v_m=u_mf_m^2\delta_m^2\to\infty\) as \(m\to\infty\);
		\item[(v)] \(\dfrac{\log(v_m)}{u_m}\to C\), \(0<C<\infty\).
	\end{itemize}
\end{assumption}

Under Assumption \ref{ass:BCFG}, \citet{BCFG2011} showed that the Type-I and Type-II error probabilities of the Bayes Oracle satisfy
\begin{align}
	t_1^{BO}
	&=
	e^{-C/2}
	\sqrt{\frac{2}{\pi v_m\log(v_m)}}
	\,
	(1+o(1)),
	\label{eq:t1_bo}
	\\
	t_2^{BO}
	&=
	\bigl(2\Phi(\sqrt{C})-1\bigr)
	(1+o(1)),
	\label{eq:t2_bo}
\end{align}
and consequently the optimal Bayes risk is given by
\begin{equation}
	\label{eq:oracle_risk}
R_m^{\mathrm{BO}}
=
m\{(1-p_m)t_1^{\mathrm{BO}}\delta_{0,m}
+
p_m t_2^{\mathrm{BO}}\delta_{A,m}\}
=
mp_m\delta_{A,m}
\{2\Phi(\sqrt C)-1\}(1+o(1)).
\end{equation}

A multiple testing procedure is said to possess the asymptotic Bayes optimality under sparsity (ABOS) property if its Bayes risk is asymptotically equivalent to the Bayes Oracle risk, that is,
\[
\frac{R_m}{R_m^{BO}}
\rightarrow
1
\quad\text{as }m\to\infty.
\]

\subsection{The Gavrilov--Benjamini--Sarkar Procedure}

For each null hypothesis $H_{0i}:\mu_i=0$, let
\[
P_i
=
2\left\{
1-\Phi\!\left(\frac{|X_i|}{\sigma}\right)
\right\}=2\bar\Phi\!\left(\frac{|X_i|}{\sigma}\right),
\qquad i=1,\ldots,m,
\]
denote the corresponding two-sided Gaussian $p$-value, where $\Phi$ and $\bar\Phi$
denote the standard normal distribution function and the survival function, respectively.

Let
\[
P_{(1)}
\le
P_{(2)}
\le
\cdots
\le
P_{(m)}
\]
denote the ordered $p$-values corresponding to the $m$ individual tests. The Gavrilov--Benjamini--Sarkar (GBS) step-down testing procedure \citep{GBS2009} employs the critical constants
\begin{equation}
	\label{eq:gbs_constants}
	c_i
	=
	\frac{i\alpha_m}
	{m+1-i(1-\alpha_m)},
	\qquad
	i=1,\ldots,m,
\end{equation}
where $\alpha_m\in(0,1)$ denotes the nominal false discovery rate level.

The GBS procedure rejects all hypotheses corresponding to
\[
P_{(1)},\ldots,P_{(R_m^{GBS})},
\]
where
\begin{equation}
	\label{eq:gbs_rule}
	R_m^{GBS}
	=
	\max
	\left\{
	i:
	P_{(j)}
	\le
	c_j
	\ \text{for all}\
	j=1,\ldots,i
	\right\},
\end{equation}
with the convention that $R_m^{GBS}=0$ if the set above is empty.

Our goal is to determine whether the GBS procedure possesses the ABOS property under the sparse asymptotic framework described above.

\section{Asymptotic Bayes Risk Analysis of the GBS Procedure}
\label{sec:ABOS-GBS-Theory}

In this section, we investigate the asymptotic Bayes risk properties of the Gavrilov--Benjamini--Sarkar (GBS) step-down testing procedure \citep{GBS2009} under the sparse asymptotic framework introduced in Section~\ref{sec:Model}. Our primary objective is to determine whether the GBS procedure attains the asymptotic Bayes optimality under sparsity (ABOS) property relative to the Bayes Oracle associated with the two-groups model \eqref{eq:two_groups_x}. To this end, we analyze separately the two components of the Bayes risk corresponding to false discoveries and false nondiscoveries. The Type-I error analysis relies on a new finite-sample inequality for shifted order statistics associated with the GBS critical constants, while the Type-II error analysis is based on a signal-crossing argument that exploits the sequential step-down structure of the procedure. These developments culminate in our main theorem establishing the ABOS property of the GBS procedure under a broad class of sparse asymptotic regimes.

For Lemmas 2--4 and Theorem~\ref{thm:gbs-abos}, we impose Assumption~\ref{ass:BCFG} and the following additional conditions on the nominal false discovery rate sequence $\{\alpha_m\}$. These conditions are specific to the diverging-signal analysis; neither they nor Assumption~\ref{ass:BCFG} are standing assumptions for the finite-sample Type-I results in Section~\ref{sec:GBS-TypeI} or the bounded-signal theorem in Section~3.3. We require
\[
\alpha_m \to 0,
\quad
\alpha_m f_m \to \infty, \quad \textrm{as} \quad m\to\infty.
\]

We further define
\[
r_{\alpha_m}
=
\frac{\alpha_m}{1-\alpha_m}.
\]

Our analysis is carried out under the following GBS calibration and compatibility conditions:
\begin{equation}
	\label{eq:gbs-calibration-condition}
	\frac{2\log(f_m/r_{\alpha_m})}{u_m}
	\to C,
	\qquad
	\frac{\log\{1/(r_{\alpha_m}\delta_m)\}}{u_m}
	\to 0,
	\qquad
	\text{as } m\to\infty,
\end{equation}
for the same constant \(0<C<\infty\) appearing in Assumption~\ref{ass:BCFG}. Under Assumption~\ref{ass:BCFG}, the two limits in \eqref{eq:gbs-calibration-condition} are equivalent: indeed, \(\log v_m=\log u_m+2\log(f_m\delta_m)\), and hence \[\frac{2\log(f_m/r_{\alpha_m})}{u_m}-\frac{\log v_m}{u_m}=\frac{2\log\{1/(r_{\alpha_m}\delta_m)\}-\log u_m}{u_m}.\] Since \(\log u_m/u_m\to0\), either displayed limit implies the other. We retain both formulations to make the relationship with the oracle asymptotics explicit. These calibration and compatibility conditions are the natural analogues of
the assumptions employed by \citet{BCFG2011} in their asymptotic Bayes optimality analysis of the Benjamini--Hochberg procedure. They ensure that
the GBS critical sequence is asymptotically aligned with the Bayes Oracle threshold under the sparse asymptotic framework.

\subsection{Control of the Type-I Error Component}
\label{sec:GBS-TypeI}
For $1\le i\le m$, let

\[
c_i
=
\frac{i\alpha_m}
{m+1-i(1-\alpha_m)}
\]

denote the GBS critical constants. We first establish a finite-sample inequality for shifted uniform order statistics; this yields Type-I control without imposing sparsity asymptotics.

\begin{lemma}
	\label{lem:gbs-shifted-null-bound}
	Let \(U_{(1)}\leq\cdots\leq U_{(n)}\) be the order statistics of \(n\)
	independent \(U(0,1)\) random variables. Let \(k\geq0\) and \(0<\alpha<1\).
	For \(\ell=1,\ldots,n\), define
	\[
	b_{\ell,k}^{(n)}(\alpha)
	=
	\frac{(k+\ell)\alpha}
	{n+k+1-(k+\ell)(1-\alpha)}.
	\]
	Then
	\[
	\sum_{j=1}^{n}
	P\left(
	U_{(\ell)}\leq b_{\ell,k}^{(n)}(\alpha),
	\ \ell=1,\ldots,j
	\right)
	\leq
	\frac{\alpha}{1-\alpha}(k+1).
	\]
\end{lemma}

\begin{proof}
	See Appendix.
\end{proof}

The significance of Lemma~\ref{lem:gbs-shifted-null-bound} lies in the fact that
it provides a finite-sample bound on the probability that a null
$p$-value occupies one of the rejection positions of the GBS procedure.
Unlike the Benjamini--Hochberg procedure, whose asymptotic analysis is
typically carried out through threshold approximations, the step-down
nature of the GBS procedure makes such an approach considerably less
transparent. Lemma~\ref{lem:gbs-shifted-null-bound} circumvents this
difficulty by directly exploiting the structure of the GBS critical
constants and thereby yields explicit control of the false-discovery component of the Bayes risk.

Let
\[
K_m=\sum_{i=1}^m \nu_i = |\mathcal H_0^c|
\]
denote the number of non-null hypotheses. Under the assumed two-groups model, \(K_m\sim \mathrm{Bin}(m,p_m)\).

\begin{corollary}
	\label{cor:type1-bound}
	Let \(V_m^{\mathrm{GBS}}\) denote the number of false discoveries made by the
	GBS procedure. Conditional on \(K_m=k\), suppose there are \(k\) non-null
	\(p\)-values and \(m-k\) null \(p\)-values. Then
	\[
	\mathbb{E}\left(V_m^{\mathrm{GBS}}\mid K_m=k\right)
	\le
	\frac{\alpha_m}{1-\alpha_m}(k+1).
	\]
	Consequently,
	\[
	\mathbb{E}\left(V_m^{\mathrm{GBS}}\right)
	\le
	\frac{\alpha_m}{1-\alpha_m}(mp_m+1).
	\]
	In particular, if \(m p_m\to\infty\) and \(\alpha_m\to0\), then
	\[
	\mathbb{E}\left(V_m^{\mathrm{GBS}}\right)
	=
	O(\alpha_m m p_m).
	\]
\end{corollary}

\begin{proof}
	See Appendix.
\end{proof}

The finite-sample bound in
Corollary~\ref{cor:type1-bound} holds for every
$m$, $p_m\in(0,1)$, and $\alpha_m\in(0,1)$
under the independent two-groups model, without
requiring any sparsity asymptotics. In particular,
it remains applicable when the expected number
of non-null hypotheses is bounded.

For the diverging-signal regime considered in
Theorem~\ref{thm:gbs-abos}, suppose that
$mp_m\to\infty$ and $\alpha_m\to0$. Then
Corollary~\ref{cor:type1-bound} implies
\[
\mathbb{E}\!\left(V_m^{\mathrm{GBS}}\right)
=O(\alpha_m mp_m).
\]
Consequently, writing
$\delta_m=\delta_{0,m}/\delta_{A,m}$,
the false-discovery component of the Bayes
risk satisfies
\[
\delta_{0,m}
\mathbb{E}\!\left(V_m^{\mathrm{GBS}}\right)
=
O\!\left(
\alpha_m\delta_m mp_m\delta_{A,m}
\right).
\]
Under the additional calibration condition
\[
\alpha_m\delta_m\to0
\qquad\text{as }m\to\infty,
\]
we therefore obtain
\[
\delta_{0,m}
\mathbb{E}\!\left(V_m^{\mathrm{GBS}}\right)
=
o\!\left(mp_m\delta_{A,m}\right).
\]
Thus, the false-discovery contribution is
asymptotically negligible relative to the
leading Bayes Oracle risk scale. Establishing
the ABOS property consequently reduces to
obtaining a sufficiently sharp asymptotic
upper bound for the false-nondiscovery
contribution. This is the objective of the
following subsection.

\subsection{Control of the Type-II Error Component}
\label{sec:GBS-TypeII}

We now turn to the false nondiscovery component of the Bayes risk. The analysis proceeds in three stages. We first establish a deterministic signal-crossing property for the alternative $p$-value distribution. We then show that this crossing is initiated sufficiently early. Finally, we translate these deterministic properties into an empirical crossing result for the ordered signal $p$-values, which yields the desired control of the Type-II component of the Bayes risk. 

Let
\[
G_m(t)
=
\mathbb{P}(P_i\le t\mid i\in\mathcal H_0^c),
\qquad 0<t<1,
\]
denote the distribution function of a \(p\)-value generated under the alternative.
The following lemma provides the deterministic signal-crossing property needed
for the empirical analysis of the ordered signal \(p\)-values.

\begin{lemma}[Deterministic signal-crossing bound]
	\label{lem:deterministic-signal-crossing}
	
	Assume that Assumption~\ref{ass:BCFG} holds. Furthermore, assume that the GBS calibration and compatibility conditions
	\eqref{eq:gbs-calibration-condition} hold, and let
	\[
	q_*
	=
	2\bar\Phi(\sqrt C).
	\]
	
	Then, for every \(0<\varepsilon<q_*\), there exists
	\(\eta_\varepsilon>0\) such that, for all sufficiently large \(m\),
	\[
	G_m(q\alpha_m p_m)
	\ge
	(1+\eta_\varepsilon)q
	\]
	uniformly over
	\[
	0<q\le q_*-\varepsilon.
	\]
\end{lemma}

\begin{proof}
	See Appendix.
\end{proof}

Lemma~\ref{lem:deterministic-signal-crossing} shows that, below the
limiting crossing point \(q_*\), the signal \(p\)-value distribution
dominates the identity map by a fixed multiplicative margin. More
precisely, for every \(\varepsilon\in(0,q_*)\), the distribution
function \(G_m\) eventually satisfies
\[
G_m(q\alpha_m p_m)\ge (1+\eta_\varepsilon)q
\]
uniformly over
\[
0<q\le q_*-\varepsilon.
\]
Consequently, throughout this range, the expected number of signal
\(p\)-values falling below a threshold of order \(q\alpha_m p_m\)
exceeds the corresponding normalized rank by a fixed multiplicative
margin. This deterministic separation is the key mechanism behind the
Type-II analysis.

To translate this deterministic phenomenon into an empirical statement
about the ordered signal \(p\)-values, it remains to verify that the
crossing mechanism is initiated sufficiently early. The following lemma
establishes this initial-crossing property.

\begin{lemma}[Initial crossing condition]
	\label{lem:initial-crossing}
	
	Assume that Assumption~\ref{ass:BCFG} holds. Furthermore, assume that the GBS calibration and compatibility conditions
	\eqref{eq:gbs-calibration-condition} hold.
	
	Then
	\[
	mp_mG_m(c_1)\to\infty
	\qquad\text{as } m\to\infty.
	\]	
\end{lemma}

\begin{proof}
	See Appendix.
\end{proof}

Lemma~\ref{lem:initial-crossing} ensures that the expected number of
signal \(p\)-values falling below the first GBS critical constant tends
to infinity. Consequently, the crossing phenomenon described in
Lemma~\ref{lem:deterministic-signal-crossing} is not merely an asymptotic
property of the bulk ranks but is already initiated at the earliest stages
of the step-down procedure.

Together, Lemmas~\ref{lem:deterministic-signal-crossing} and
\ref{lem:initial-crossing} establish that the signal \(p\)-value
distribution exhibits a persistent excess over the identity map below the
limiting crossing point \(q_*\), and that this excess is present from the
very beginning of the ranking sequence. It remains to translate these
deterministic properties into a statement about the empirical behavior of
the ordered signal \(p\)-values. The following lemma shows that, with
probability tending to one, the ordered signal \(p\)-values cross the
corresponding GBS critical constants throughout a substantial range of
ranks. This empirical crossing property yields the desired control of the
false-nondiscovery component of the Bayes risk.

\begin{lemma}[Empirical signal-crossing bound]
	\label{lem:empirical-signal-crossing}
	
	Assume that Assumption~\ref{ass:BCFG} holds. Furthermore, assume that the GBS calibration and compatibility conditions
\eqref{eq:gbs-calibration-condition} hold.

	Let
	\[
	q_* = 2\bar\Phi(\sqrt C),
	\]
	and fix \(\varepsilon\in(0,q_*)\). Define
	\[
	r_m=\lfloor(q_*-\varepsilon)mp_m\rfloor.
	\]
	
	Then
	\[
	\mathbb P\!\left(R_m^{\mathrm{GBS}}\ge r_m\right)\to1
	\quad\textrm{as}\quad m\to\infty.
	\]
	
	Consequently,
	\[
	\mathbb E\!\left(T_m^{\mathrm{GBS}}\right)
	\le
	mp_m(1-q_*+\varepsilon)+o(mp_m).
	\]
\end{lemma}

\begin{proof}
	See Appendix.
\end{proof}

Lemma~\ref{lem:empirical-signal-crossing} implies that, with probability
tending to one, the GBS procedure rejects at least
\[
r_m=\lfloor (q_*-\varepsilon)m p_m\rfloor
\]
hypotheses. Consequently, after accounting for the (asymptotically negligible)
false discoveries established in Corollary~\ref{cor:type1-bound}, the number of
missed signals is asymptotically bounded above by
\[
K_m-r_m
\]
up to lower-order terms. Since \(K_m/(m p_m)\to1\) as \(m\to\infty\), it follows
from Lemma~\ref{lem:empirical-signal-crossing} that
\[
\mathbb{E}(T_m^{\mathrm{GBS}})
\le
m p_m(1-q_*+\varepsilon)+o(m p_m).
\]

Recalling that
\[
q_*=2\bar\Phi(\sqrt C),
\]
we obtain
\[
\mathbb{E}(T_m^{\mathrm{GBS}})
\le
m p_m\bigl(2\Phi(\sqrt C)-1+\varepsilon\bigr)
+o(m p_m).
\]
For each fixed \(\varepsilon\in(0,q_*)\), divide the preceding inequality by
\(mp_m\) and take the limit superior as \(m\to\infty\). This yields
\[
\limsup_{m\to\infty}
\frac{\mathbb{E}(T_m^{\mathrm{GBS}})}{mp_m}
\le 2\Phi(\sqrt C)-1+\varepsilon.
\]
Since this holds for every such fixed \(\varepsilon\), letting
\(\varepsilon\downarrow0\) gives
\[
\mathbb{E}(T_m^{\mathrm{GBS}})
\le
m p_m\bigl(2\Phi(\sqrt C)-1\bigr)
+o(m p_m).
\]
The order of limits is important: no uniformity of the remainder in
\(\varepsilon\) is required.
Multiplying by the false-nondiscovery loss \(\delta_{A,m}\), we get
\[
\delta_{A,m}\mathbb{E}(T_m^{\mathrm{GBS}})
\le
m p_m\delta_{A,m}
\bigl(2\Phi(\sqrt C)-1\bigr)
+o(m p_m\delta_{A,m}).
\]

On the other hand, under Assumption~\ref{ass:BCFG}, the Bayes Oracle satisfies
\[
R_m^{\mathrm{BO}}
=
m p_m\delta_{A,m}
\bigl(2\Phi(\sqrt C)-1\bigr)(1+o(1)).
\]
Therefore,
\[
\delta_{A,m}\mathbb{E}(T_m^{\mathrm{GBS}})
\le
R_m^{BO}
+
o(mp_m\delta_{A,m}).
\]

Since
\[
\delta_{0,m}\mathbb{E}(V_m^{\mathrm{GBS}})
=
o(mp_m\delta_{A,m}),
\]
it follows that
\[
R_m^{GBS}
\le
R_m^{BO}
+
o(mp_m\delta_{A,m}).
\]

The following theorem shows that this upper bound is in fact asymptotically sharp.

\subsection{Main Theoretical Results}
\label{sec:ABOS-GBS-Main-Result}

We are now in a position to establish the asymptotic Bayes optimality
under sparsity of the Gavrilov--Benjamini--Sarkar step-down procedure \citep{GBS2009}.

\begin{theorem}[ABOS of the GBS procedure]
	\label{thm:gbs-abos}
	Assume that Assumption~\ref{ass:BCFG} holds. Furthermore, suppose that
	\[
	\alpha_m\to0,
	\qquad
	\alpha_m\delta_m\to0,
	\qquad
	\alpha_m f_m\to\infty,
	\qquad \text{as } m\to\infty,
	\]
	and that the GBS calibration and compatibility conditions
	\eqref{eq:gbs-calibration-condition} hold. Then
	\[
	R_m^{\mathrm{GBS}}
	=
	R_m^{\mathrm{BO}}(1+o(1)).
	\]
	Equivalently,
	\[
	\frac{R_m^{\mathrm{GBS}}}{R_m^{\mathrm{BO}}}
	\to 1
	\qquad \text{as } m\to\infty.
	\]
Consequently, the Gavrilov–Benjamini–Sarkar step-down procedure possesses the asymptotic Bayes optimality under sparsity (ABOS) property.
\end{theorem}

\begin{proof}
Let \(V_m^{GBS}\) and \(T_m^{GBS}\) denote, respectively, the numbers of false discoveries and false nondiscoveries made by the GBS procedure. Then its Bayes risk is
\[
R_m^{GBS}
=
\delta_{0,m}\mathbb{E}(V_m^{GBS})
+
\delta_{A,m}\mathbb{E}(T_m^{GBS}).
\]	

By Corollary~\ref{cor:type1-bound},
	\[
	\mathbb{E}(V_m^{GBS})
	=
	O(\alpha_m m p_m).
	\]
	Therefore, writing \(\delta_m=\delta_{0,m}/\delta_{A,m}\),
	\[
	\delta_{0,m}\mathbb{E}(V_m^{GBS})
	=
	O(\alpha_m\delta_m\, m p_m\delta_{A,m})
	=
	o(m p_m\delta_{A,m}),
	\]
	since \(\alpha_m\delta_m\to0\) as \(m\to\infty\).
	
	On the other hand, applying Lemma~\ref{lem:empirical-signal-crossing}
for each fixed \(\varepsilon\in(0,q_*)\), taking the limit superior
as \(m\to\infty\), and then letting \(\varepsilon\downarrow0\)
yields
	\[
	\mathbb{E}(T_m^{GBS})
	\le
	m p_m\bigl(2\Phi(\sqrt C)-1\bigr)
	+o(m p_m).
	\]
Multiplying both sides by $\delta_{A,m}$,
	\[
	\delta_{A,m}\mathbb{E}(T_m^{GBS})
	\le
	m p_m\delta_{A,m}
	\bigl(2\Phi(\sqrt C)-1\bigr)
	+o(m p_m\delta_{A,m}).
	\]
	Combining the Type-I and Type-II bounds gives
	\[
	R_m^{GBS}
	\le
	m p_m\delta_{A,m}
	\bigl(2\Phi(\sqrt C)-1\bigr)
	+o(m p_m\delta_{A,m}).
	\]
	By the asymptotic expansion of the Bayes Oracle risk in \eqref{eq:oracle_risk},
	\[
	R_m^{BO}
	=
	m p_m\delta_{A,m}
	\bigl(2\Phi(\sqrt C)-1\bigr)(1+o(1)).
	\]
	Thus
	\[
	R_m^{GBS}
	\le
	R_m^{BO}(1+o(1)),
	\]
	for all sufficiently large \(m\). Since the Bayes Oracle minimizes the Bayes risk under the two-groups model,
	\[
	R_m^{GBS}\ge R_m^{BO},
	\]
	for all $m\ge 1$.
	Therefore, for all sufficiently large \(m\), we have
	\[
	1
	\le
	\frac{R_m^{GBS}}{R_m^{BO}}
	\le
	1+o(1),
	\]
whence
	\[
	\frac{R_m^{GBS}}{R_m^{BO}}
	\to 1 \ \textrm{as} \ m\to\infty.
	\]
	Hence the GBS procedure possesses the ABOS property.
\end{proof}

\begin{theorem}[ABOS of GBS with bounded expected signal counts]
\label{thm:gbs-abos-bounded}
Under the independent Gaussian two-groups model, suppose that
\[
p_m\to0,\qquad mp_m\to\lambda\in(0,\infty),\qquad
\frac{u_m}{2\log m}\to\frac1C,\quad C\in(0,\infty),
\]
and that the loss ratio is fixed, $\delta_m=\delta\in(0,\infty)$.
If $\alpha_m\to0$ and $\log(1/\alpha_m)=o(\log m)$, then
\[
\frac{R_m^{\mathrm{GBS}}}{R_m^{\mathrm{BO}}}\longrightarrow1.
\]
\end{theorem}

\begin{proof}
The first GBS critical constant is
$c_{1,m}=\alpha_m/(m+\alpha_m)$. If $P_i\le c_{1,m}$
and $P_i=P_{(j)}$, then for every $\ell\le j$,
\[
P_{(\ell)}\le P_{(j)}\le c_{1,m}\le c_{\ell,m}.
\]
Thus the GBS step-down rule rejects hypothesis $i$. In particular,
for each non-null coordinate, the event of nonrejection is contained in
$\{P_i>c_{1,m}\}$. Writing $G_m$ for the alternative $p$-value
cdf, linearity of expectation yields the direct GBS bound
\begin{equation}\label{eq:bounded-gbs-type2}
\mathbb E(T_m^{\mathrm{GBS}})
\le mp_m\{1-G_m(c_{1,m})\}.
\end{equation}
Let $z_m=\Phi^{-1}(1-c_{1,m}/2)$. Since
$\log(1/\alpha_m)=o(\log m)$, we have
$\log(1/c_{1,m})=\log m+o(\log m)$. The standard normal
upper-tail quantile asymptotic therefore gives
\[
z_m^2=2\log(1/c_{1,m})(1+o(1)),\qquad
\frac{z_m^2}{1+u_m}\longrightarrow C.
\]
Under the alternative, $X_i\sim N(0,1+u_m)$, whence
\[
1-G_m(c_{1,m})
=2\Phi\!\left(\frac{z_m}{\sqrt{1+u_m}}\right)-1
\longrightarrow 2\Phi(\sqrt C)-1.
\]
It follows from \eqref{eq:bounded-gbs-type2} that
\begin{equation}\label{eq:bounded-type2-result}
\mathbb E(T_m^{\mathrm{GBS}})
\le mp_m\{2\Phi(\sqrt C)-1+o(1)\}.
\end{equation}

For the Type-I component, Corollary~\ref{cor:type1-bound}
provides the finite-sample inequality
\[
\mathbb E(V_m^{\mathrm{GBS}})
\le\frac{\alpha_m}{1-\alpha_m}(mp_m+1).
\]
As $mp_m\to\lambda>0$, $\alpha_m\to0$, and
$\delta_m=\delta$, this implies
$\delta_{0,m}\mathbb E(V_m^{\mathrm{GBS}})
=o(mp_m\delta_{A,m})$. Combining this with
\eqref{eq:bounded-type2-result} gives
\[
R_m^{\mathrm{GBS}}
\le mp_m\delta_{A,m}\{2\Phi(\sqrt C)-1+o(1)\}.
\]
Finally, $f_m=(1-p_m)/p_m\sim m/\lambda$ and
$u_m\sim(2/C)\log m$, so
$\log(u_mf_m^2\delta^2)/u_m\to C$.
The Bayes Oracle expansion of \citet{BCFG2011} therefore yields
\[
R_m^{\mathrm{BO}}
=mp_m\delta_{A,m}\{2\Phi(\sqrt C)-1\}(1+o(1)).
\]
Since $R_m^{\mathrm{GBS}}\ge R_m^{\mathrm{BO}}$ for every $m$,
the desired ratio converges to one.
\end{proof}

\begin{corollary}[ABOS throughout polynomial sparsity]
\label{cor:gbs-polynomial}
Suppose the independent Gaussian two-groups model holds and, for
some fixed $\beta\in(0,1]$,
\[
p_m\asymp m^{-\beta},\qquad
\frac{u_m}{\log m}\longrightarrow\frac{2\beta}{C},
\quad C\in(0,\infty),\qquad
\delta_m=\delta\in(0,\infty).
\]
If $\alpha_m\to0$ and
$\log(1/\alpha_m)=o(\log m)$, then GBS is ABOS.
\end{corollary}
\begin{proof}
For $0<\beta<1$, $mp_m\to\infty$, and the assumptions
of Theorem~\ref{thm:gbs-abos} follow from
$\log f_m=\beta\log m+O(1)$,
$\log(1/\alpha_m)=o(\log m)$, and the stated variance
scaling. For $\beta=1$, $mp_m$ is bounded above and bounded away from zero, but need not converge. Although Theorem~\ref{thm:gbs-abos-bounded} is stated under $mp_m\to\lambda$, its proof applies directly under these two-sided bounds: the first-critical-value Type-II bound and the oracle asymptotics use $\log f_m=\log m+O(1)$, while the finite-sample Type-I bound is negligible relative to $mp_m\delta_{A,m}$. Thus no convergence of $mp_m$ is needed for this corollary.
\end{proof}

\begin{remark}
Theorem~\ref{thm:gbs-abos-bounded} is based directly on the
first GBS critical constant and the step-down rejection rule;
it does not compare GBS to a different testing procedure.
Unlike the proof for diverging $mp_m$, it does not require
concentration of the number of signals or uniform crossing
of ordered alternative $p$-values.
\end{remark}

Theorems~\ref{thm:gbs-abos} and~\ref{thm:gbs-abos-bounded}, together with
Corollary~\ref{cor:gbs-polynomial}, show that ABOS persists from
regimes with a diverging expected number of signals to the boundary
of polynomial sparsity, where the expected signal count remains
bounded away from zero and infinity. This extension does not
require localization of the random GBS rejection threshold.

\section{Discussion}
\label{sec:Discussion}

Our results establish the asymptotic Bayes optimality under sparsity
(ABOS) of the Gavrilov--Benjamini--Sarkar (GBS) step-down procedure
under the independent Gaussian two-groups model. The theory covers
polynomial sparsity $p_m\asymp m^{-\beta}$ for every fixed
$0<\beta\leq 1$, subject to the stated signal-strength and
calibration conditions. In particular, the boundary case $\beta=1$
is accommodated without assuming concentration of the signal count.

The methodological contribution is a direct Bayes-risk analysis
that avoids localization of the random rejection threshold. A
finite-sample inequality for shifted uniform order statistics
controls false discoveries, while the step-down structure yields
two complementary ways to bound missed signals: a uniform
signal-crossing argument when $mp_m\to\infty$, and a bound based
on the first critical constant when the expected signal count is
bounded. These tools may be useful beyond the specific GBS
procedure considered here.

The present findings also complement the sharp asymptotic
minimaxity theory for the GBS procedure developed in
\citet{GhoshGBSMinimax2026}. Whereas the present paper establishes
asymptotic equivalence between the Bayes risk of GBS and that of
the Bayes Oracle under a specified two-groups prior, the
minimaxity analysis evaluates the worst-case frequentist risk
over sparse parameter classes without imposing a prior
distribution on the unknown mean vector. In particular,
\citet{GhoshGBSMinimax2026} established sharp asymptotic
minimaxity of GBS under both normalized Hamming loss and the
combined false discovery proportion and false nondiscovery
proportion loss over classical beta-min parameter classes.
For heterogeneous signal-strength classes, that work further
established a sharp asymptotic upper bound for normalized
Hamming risk and attainment of the corresponding sharp minimax
benchmark under the combined proportion-based loss. These
results provide complementary Bayesian and frequentist
perspectives on the optimality of GBS. Although both analyses
exploit the step-down structure through shifted order-statistic
bounds and signal-crossing arguments, they address different
risk criteria and parameter regimes. Neither optimality
property follows directly from the other.

A further decision-theoretic property of the GBS procedure is
its finite-sample admissibility under the vector-valued loss
considered by \citet{GhoshChakrabartiAdmissibility2026}. Specifically,
under independence, the classical GBS procedure is a special
case of the residual-based step-down procedures studied in
that work, whose admissibility theory accommodates arbitrary
covariance dependence. Thus, within their respective loss
formulations and assumptions, the GBS methodology admits
several complementary theoretical justifications: frequentist
FDR control, finite-sample admissibility, asymptotic Bayes
optimality under sparsity, and sharp asymptotic minimaxity.
The near-Oracle risk behavior reported in recent simulation
studies \citep{GhoshChakrabartiBSD2026} provides additional
empirical motivation for investigating these properties,
although such numerical findings do not themselves establish
asymptotic optimality under dependence.

An important direction for future research is to extend the present decision-theoretic analysis to multiple testing procedures that explicitly account for dependence among the test statistics. Unlike the independent setting considered here, covariance dependence can fundamentally alter the structure of the Bayes Oracle and complicate the characterization of its risk. For covariance-adaptive step-down procedures, an additional difficulty arises from the sequential recomputation of residual statistics along data-dependent rejection paths, which changes both the evidence available for subsequent decisions and its dependence structure. Recent empirical investigations \citep{GhoshChakrabartiMRDGBS2026,GhoshChakrabartiBSD2026} suggest that favorable risk and signal-recovery performance may persist under structured dependence. Establishing corresponding asymptotic optimality guarantees, however, requires new theoretical tools capable of accounting for this evolving dependence.

Another promising direction is to investigate the extent to which the techniques developed here extend beyond the GBS procedure. In particular, it would be interesting to identify general conditions on the critical sequences of step-down procedures that ensure asymptotic Bayes optimality under sparsity. Such an investigation could clarify which features of the GBS critical constants are essential for optimality and which admit broader generalizations.

\appendix  
\section*{Appendix}
\label{app}

\section{Proofs of the Main Theoretical Results}

This appendix contains the proofs of Lemmas~\ref{lem:gbs-shifted-null-bound}--\ref{lem:empirical-signal-crossing} and Corollary~\ref{cor:type1-bound} used in the
development of the asymptotic Bayes optimality analysis in Section~\ref{sec:ABOS-GBS-Theory}.

\subsection*{Proof of Lemma~\ref{lem:gbs-shifted-null-bound}}
\begin{proof}
	For \(j=1,\ldots,n\), define the events
	\begin{equation}\notag
		A_j^{(n,k,\alpha)}
		=
		\left\{
		U_{(\ell)}\leq b_{\ell,k}^{(n)}(\alpha),
		\ \ell=1,\ldots,j
		\right\}.
	\end{equation}
	Let
	\begin{equation}
		\label{eq:Lem1_Sn_Definition}
		S_n(k,\alpha)
		=
		\sum_{j=1}^{n}P\left(A_j^{(n,k,\alpha)}\right).
	\end{equation}
	We prove by induction on \(n\) that
	\begin{equation}
		\label{eq:Lem1_Induction_Claim}
		S_n(k,\alpha)
		\leq
		\frac{\alpha}{1-\alpha}(k+1),
	\end{equation}
	for all \(n\geq 1\), \(k\geq0\), and \(0<\alpha<1\).
	
	For \(n=1\),
	\begin{equation}\notag
		\label{eq:Lem1_Base_Step}
		S_1(k,\alpha)
		=
		P\left(
		U_{(1)}
		\le
		b_{1,k}^{(1)}(\alpha)
		\right)
		=
		b_{1,k}^{(1)}(\alpha),
	\end{equation}
	because \(U_{(1)}\), being a single observation drawn from \(U(0,1)\), is uniformly distributed over \((0,1)\). Now
	\begin{equation}\notag
		\label{eq:Lem1_Base_Boundary}
		b_{1,k}^{(1)}(\alpha)
		=
		\frac{(k+1)\alpha}
		{1+k+1-(k+1)(1-\alpha)}
		=
		\frac{(k+1)\alpha}
		{1+(k+1)\alpha}.
	\end{equation}
	Since
	\begin{equation}\notag
		\label{eq:Lem1_Base_Inequality}
		\frac{(k+1)\alpha}
		{1+(k+1)\alpha}
		\leq
		\frac{\alpha}{1-\alpha}(k+1),
	\end{equation}
	the claim holds for \(n=1\).
	
	Now assume that the result holds for \(n-1\), where \(n\geq2\). We prove it
	for \(n\). Let
	\begin{equation}\notag
		\label{eq:Lem1_b1_Definition}
		b_1=b_{1,k}^{(n)}(\alpha).
	\end{equation}
	Since
	\begin{equation}\notag
		\label{eq:Lem1_b1_Odds}
		\frac{b_1}{1-b_1}
		=
		\frac{(k+1)\alpha}{n},
	\end{equation}
	we have
	\begin{equation}
		\label{eq:Lem1_b1_Formula}
		b_1
		=
		\frac{(k+1)\alpha}{n+(k+1)\alpha}.
	\end{equation}
	
	Condition on \(U_{(1)}=u\), where \(0<u<b_1\). Given \(U_{(1)}=u\), the
	remaining \(n-1\) order statistics, after the affine transformation
	\begin{equation}\notag
		\label{eq:Lem1_Affine_Transformation}
		W_i=\frac{U_i-u}{1-u},
	\end{equation}
	are distributed as the order statistics of \(n-1\) independent \(U(0,1)\)
	random variables.
	
	Therefore, if the first crossing condition \(U_{(1)}\le b_1\) is satisfied, the
	remaining crossing conditions may be rewritten in terms of the transformed order
	statistics \(W_{(1)},\ldots,W_{(n-1)}\). More precisely, for
	\(\ell=1,\ldots,n-1\), the residual boundary is
	\begin{equation}\notag
		\label{eq:Lem1_Residual_Boundary}
		\frac{b_{\ell+1,k}^{(n)}(\alpha)-u}{1-u}.
	\end{equation}
	
	Thus, conditioning on \(U_{(1)}=u\), we obtain
	\begin{equation}
		\label{eq:Lem1_Sn_Recursive_Preliminary}
		S_n(k,\alpha)
		=
		\int_0^{b_1}
		n(1-u)^{n-1}
		\left\{
		1+
		\widetilde S_{n-1}(u)
		\right\}
		\,du,
	\end{equation}
	where
	\begin{equation}\notag
		\label{eq:Lem1_TildeS_Definition}
		\widetilde S_{n-1}(u)
		=
		\sum_{j=1}^{n-1}
		P\left(
		W_{(\ell)}
		\le
		\frac{b_{\ell+1,k}^{(n)}(\alpha)-u}{1-u},
		\ \ell=1,\ldots,j
		\right).
	\end{equation}
	
	We now show that the residual sum \(\widetilde S_{n-1}(u)\) is again of the same
	form as \(S_{n-1}\), but with transformed parameters. Put
	\begin{equation}\notag
		x=\frac{u}{1-u}.
	\end{equation}
	Define
	\begin{equation}\notag
		\alpha_u
		=
		\frac{\alpha+x}{1+x}
	\end{equation}
	and
	\begin{equation}\notag
		k_u
		=
		\frac{(k+1)\alpha-nx}{\alpha+x}.
	\end{equation}
	Since
	\begin{equation}\notag
		0<u<b_1
		=
		\frac{(k+1)\alpha}
		{n+(k+1)\alpha},
	\end{equation}
	and the map \(u\mapsto u/(1-u)\) is strictly increasing on \((0,1)\),
	\begin{equation}\notag
		0<x=\frac{u}{1-u}
		<
		\frac{b_1}{1-b_1}
		=
		\frac{(k+1)\alpha}{n}.
	\end{equation}
	Consequently,
	\begin{equation}\notag
		(k+1)\alpha-nx>0.
	\end{equation}
	Since \(\alpha+x>0\), it follows that
	\begin{equation}\notag
		k_u
		=
		\frac{(k+1)\alpha-nx}{\alpha+x}
		>0.
	\end{equation}
	Also \(0<\alpha_u<1\).
	
	We claim that for every \(\ell=1,\ldots,n-1\),
	\begin{equation}
		\label{eq:Lem1_Boundary_Transformation_Claim}
		\frac{b_{\ell+1,k}^{(n)}(\alpha)-u}{1-u}
		=
		b_{\ell,k_u}^{(n-1)}(\alpha_u).
	\end{equation}
	To verify this claim, first observe that
	\begin{equation}
	\label{eq:Lem1_Odds_Ratio}
		\frac{b_{\ell,k}^{(n)}(\alpha)}
		{1-b_{\ell,k}^{(n)}(\alpha)}
		=
		\frac{(k+\ell)\alpha}{n+1-\ell}.
	\end{equation}
	Again, observe that
	\begin{equation}\notag
		1-\frac{b_{\ell+1,k}^{(n)}(\alpha)-u}{1-u}
		=
		\frac{1-b_{\ell+1,k}^{(n)}(\alpha)}
		{1-u}.
	\end{equation}
	Hence
	\begin{eqnarray}
	\label{eq:Lem1_Residual_Odds}
		\frac{
			\left(b_{\ell+1,k}^{(n)}(\alpha)-u\right)/(1-u)
		}{
			1-\left(b_{\ell+1,k}^{(n)}(\alpha)-u\right)/(1-u)
		}
		&=&
		\frac{
			\left(b_{\ell+1,k}^{(n)}(\alpha)-u\right)/(1-u)
		}{
			\left(1-b_{\ell+1,k}^{(n)}(\alpha)\right)/(1-u)
		}
		\nonumber\\
		&=&
		\frac{
			b_{\ell+1,k}^{(n)}(\alpha)-u
		}{
			1-b_{\ell+1,k}^{(n)}(\alpha)
		}.
	\end{eqnarray}
	Using \eqref{eq:Lem1_Odds_Ratio} with \(\ell\) replaced by \(\ell+1\), we get
	\begin{equation}
		\label{eq:Lem1_Shifted_Odds_Ratio}
		\frac{b_{\ell+1,k}^{(n)}(\alpha)}
		{1-b_{\ell+1,k}^{(n)}(\alpha)}
		=
		\frac{(k+\ell+1)\alpha}{n-\ell}.
	\end{equation}
	Writing \(x=u/(1-u)\) in \eqref{eq:Lem1_Residual_Odds} and using
	\eqref{eq:Lem1_Shifted_Odds_Ratio}, we obtain
	\begin{equation}\notag
		\frac{b_{\ell+1,k}^{(n)}(\alpha)-u}
		{1-b_{\ell+1,k}^{(n)}(\alpha)}
		=
		\frac{
			(k+\ell+1)\alpha-(n-\ell)x
		}
		{(1+x)(n-\ell)}.
	\end{equation}
	
	Again, using the definition of \(b_{\ell+1,k}^{(n)}(\alpha)\), we have
	\begin{equation}
	\label{eq:Lem1_Boundary_Definition}
		b_{\ell+1,k}^{(n)}(\alpha)
		=
		\frac{(k+\ell+1)\alpha}
		{n+k+1-(k+\ell+1)(1-\alpha)}.
	\end{equation}
	Since
	\begin{equation}\notag
		n+k+1-(k+\ell+1)(1-\alpha)
		=
		n-\ell+(k+\ell+1)\alpha,
	\end{equation}
	using \eqref{eq:Lem1_Boundary_Definition}, we may write
	\begin{equation}
	\label{eq:Lem1_Boundary_Simplified}
		b_{\ell+1,k}^{(n)}(\alpha)
		=
		\frac{(k+\ell+1)\alpha}
		{n-\ell+(k+\ell+1)\alpha}.
	\end{equation}
	Therefore,
	\begin{equation}
	\label{eq:Lem1_Boundary_Complement}
		1-b_{\ell+1,k}^{(n)}(\alpha)
		=
		\frac{n-\ell}
		{n-\ell+(k+\ell+1)\alpha}.
	\end{equation}
	Hence, combining
	\eqref{eq:Lem1_Boundary_Simplified}
	and
	\eqref{eq:Lem1_Boundary_Complement}, we obtain
	\begin{eqnarray}
	\label{eq:Lem1_Residual_Odds_Expansion}
		\frac{b_{\ell+1,k}^{(n)}(\alpha)-u}
		{1-b_{\ell+1,k}^{(n)}(\alpha)}
		&=&
		\frac{
			(k+\ell+1)\alpha
			-
			u\{n-\ell+(k+\ell+1)\alpha\}
		}
		{n-\ell}.
	\end{eqnarray}
	Since
	\begin{equation}\notag
		u=\frac{x}{1+x},
	\end{equation}
	using \eqref{eq:Lem1_Residual_Odds_Expansion}, we get
	\begin{equation}	\label{eq:Lem1_Residual_Odds_Formula_Second}
		\frac{b_{\ell+1,k}^{(n)}(\alpha)-u}
		{1-b_{\ell+1,k}^{(n)}(\alpha)}
		=
		\frac{
			(k+\ell+1)\alpha-(n-\ell)x
		}
		{(1+x)(n-\ell)}.
	\end{equation}
	
	On the other hand, using the definitions of \(k_u\) and \(\alpha_u\), we obtain
	\begin{eqnarray}	\label{eq:Lem1_ku_alphau_identity}
		\frac{(k_u+\ell)\alpha_u}{n-\ell}
		&=&
		\frac{
			\left(
			\frac{(k+1)\alpha-nx}{\alpha+x}
			+\ell
			\right)
			\frac{\alpha+x}{1+x}
		}
		{n-\ell}
		\nonumber\\
		&=&
		\frac{
			(k+\ell+1)\alpha-(n-\ell)x
		}
		{(1+x)(n-\ell)}.
	\end{eqnarray}
	Combining
	\eqref{eq:Lem1_Residual_Odds},
	\eqref{eq:Lem1_Residual_Odds_Formula_Second},
	\eqref{eq:Lem1_ku_alphau_identity},
	and \eqref{eq:Lem1_Odds_Ratio}
	with \(n\) replaced by \(n-1\), \(k\) by \(k_u\), and \(\alpha\) by \(\alpha_u\),
	we obtain
	\begin{equation}
	\label{eq:Lem1_Final_Odds_Equality}
		\frac{
			\left(b_{\ell+1,k}^{(n)}(\alpha)-u\right)/(1-u)
		}{
			1-\left(b_{\ell+1,k}^{(n)}(\alpha)-u\right)/(1-u)
		}
		=
		\frac{
			b_{\ell,k_u}^{(n-1)}(\alpha_u)
		}{
			1-b_{\ell,k_u}^{(n-1)}(\alpha_u)
		}.
	\end{equation}
	Since the map \(t\mapsto t/(1-t)\) is one-to-one on \((0,1)\), it follows from
	\eqref{eq:Lem1_Final_Odds_Equality} that
	\begin{equation}\notag
		\frac{b_{\ell+1,k}^{(n)}(\alpha)-u}{1-u}
		=
		b_{\ell,k_u}^{(n-1)}(\alpha_u).
	\end{equation}
	This establishes our claim \eqref{eq:Lem1_Boundary_Transformation_Claim}.
	
	Consequently,
	\begin{equation}
		\label{eq:Lem1_TildeS_Equals_S}
		\widetilde S_{n-1}(u)=S_{n-1}(k_u,\alpha_u).
	\end{equation}
	Combining \eqref{eq:Lem1_Sn_Recursive_Preliminary} and
	\eqref{eq:Lem1_TildeS_Equals_S}, we obtain
	\begin{equation}
		\label{eq:Lem1_Sn_Recursive_Final}
		S_n(k,\alpha)
		=
		\int_0^{b_1}
		n(1-u)^{n-1}
		\left\{
		1+S_{n-1}(k_u,\alpha_u)
		\right\}
		\,du.
	\end{equation}
	By the induction hypothesis,
	\begin{equation}\notag
		S_{n-1}(k_u,\alpha_u)
		\leq
		\frac{\alpha_u}{1-\alpha_u}(k_u+1).
	\end{equation}
	Now
	\begin{equation}\notag
		1-\alpha_u
		=
		1-\frac{\alpha+x}{1+x}
		=
		\frac{1-\alpha}{1+x}.
	\end{equation}
	Hence
	\begin{equation}\notag
		\frac{\alpha_u}{1-\alpha_u}
		=
		\frac{\alpha+x}{1-\alpha}.
	\end{equation}
	Also,
	\begin{equation}\notag
		k_u+1
		=
		\frac{(k+1)\alpha-nx}{\alpha+x}+1
		=
		\frac{(k+2)\alpha-(n-1)x}{\alpha+x}.
	\end{equation}
	Therefore,
	\begin{equation}\notag
		\frac{\alpha_u}{1-\alpha_u}(k_u+1)
		=
		\frac{(k+2)\alpha-(n-1)x}{1-\alpha}.
	\end{equation}
	It follows that
	\begin{eqnarray}
    \label{eq:Lem1_One_Plus_Bound}
		1+S_{n-1}(k_u,\alpha_u)
		&\leq&
		1+
		\frac{(k+2)\alpha-(n-1)x}{1-\alpha}
		\nonumber\\
		&=&
		\frac{
			1+(k+1)\alpha-(n-1)x
		}
		{1-\alpha}.
	\end{eqnarray}
	
	Consequently, since \(x=u/(1-u)\), we obtain from
	\eqref{eq:Lem1_Sn_Recursive_Final} and \eqref{eq:Lem1_One_Plus_Bound} that
	\begin{equation}
		\label{eq:Lem1_Sn_Integral_Bound}
		S_n(k,\alpha)
		\leq
		\frac{1}{1-\alpha}
		\int_0^{b_1}
		n(1-u)^{n-1}
		\left\{
		1+(k+1)\alpha-(n-1)\frac{u}{1-u}
		\right\}
		\,du.
	\end{equation}
	Let
	\begin{equation}\notag
		A=(k+1)\alpha.
	\end{equation}
	Then, by \eqref{eq:Lem1_b1_Formula},
	\begin{equation}\notag
		b_1=\frac{A}{n+A}.
	\end{equation}
	The integral in \eqref{eq:Lem1_Sn_Integral_Bound} is
	\begin{equation}
		\label{eq:Lem1_I_Definition}
		I
		=
		\int_0^{b_1}
		n(1-u)^{n-1}
		\left\{
		1+A-(n-1)\frac{u}{1-u}
		\right\}
		\,du.
	\end{equation}
	Rewrite the integrand as
	\begin{equation}\notag
		n(1-u)^{n-1}(1+A)
		-
		n(n-1)u(1-u)^{n-2}.
	\end{equation}
	Since
	\begin{equation}\notag
		\frac{d}{du}\left\{nu(1-u)^{n-1}\right\}
		=
		n(1-u)^{n-1}
		-
		n(n-1)u(1-u)^{n-2},
	\end{equation}
	we get
	\begin{equation}\notag
		I
		=
		\left[nu(1-u)^{n-1}\right]_{0}^{b_1}
		+
		A\int_0^{b_1}n(1-u)^{n-1}\,du.
	\end{equation}
	Therefore,
	\begin{equation}\notag
		I
		=
		nb_1(1-b_1)^{n-1}
		+
		A\left\{1-(1-b_1)^n\right\}.
	\end{equation}
	Since
	\begin{equation}\notag
		b_1=\frac{A}{n+A},
		\qquad
		1-b_1=\frac{n}{n+A},
	\end{equation}
	we have
	\begin{eqnarray}\notag
		nb_1(1-b_1)^{n-1}
		&=&
		n\frac{A}{n+A}
		\left(\frac{n}{n+A}\right)^{n-1}
		\nonumber\\
		&=&
		A\left(\frac{n}{n+A}\right)^n.\nonumber
	\end{eqnarray}
	Thus
	\begin{eqnarray}
		\label{eq:Lem1_I_Equals_A}
		I
		&=&
		A\left(\frac{n}{n+A}\right)^n
		+
		A\left\{
		1-
		\left(\frac{n}{n+A}\right)^n
		\right\}
		\nonumber\\
		&=&
		A.
	\end{eqnarray}
	Hence, by \eqref{eq:Lem1_Sn_Integral_Bound},
	\eqref{eq:Lem1_I_Definition}, and \eqref{eq:Lem1_I_Equals_A},
	\begin{equation}\notag
		S_n(k,\alpha)
		\leq
		\frac{A}{1-\alpha}
		=
		\frac{\alpha}{1-\alpha}(k+1).
	\end{equation}
	This is precisely \eqref{eq:Lem1_Induction_Claim}. Hence the induction is complete,
	and the lemma follows.
\end{proof}

\subsection*{Proof of Corollary~\ref{cor:type1-bound}}

\begin{proof}
	Let \(K_m=k\), and put \(n=m-k\). If \(k=m\), there are no true null hypotheses, so \(V_m^{\mathrm{GBS}}=0\) almost surely and the claimed bound is immediate. Henceforth assume \(n\ge1\). Conditional on \(K_m=k\), there are \(k\)
	non-null \(p\)-values and \(n\) null \(p\)-values. Let
	\begin{equation*}
		P_{(1)}\le \cdots \le P_{(n)}
	\end{equation*}
	denote the order statistics of the null \(p\)-values. Under independence,
	these are the order statistics of \(n\) independent \(U(0,1)\) random
	variables.
	
	Fix \(j=1,\ldots,n\). Suppose that the GBS procedure makes at least \(j\)
	false discoveries. Then at least the \(j\) smallest null \(p\)-values,
	\(P_{(1)},\ldots,P_{(j)}\), must be rejected.
	
	For \(\ell=1,\ldots,j\), let \(R_\ell\) denote the rank of \(P_{(\ell)}\)
	among all \(m\) ordered \(p\)-values. Since there are exactly \(k\) non-null
	\(p\)-values, at most \(k\) non-null \(p\)-values can be smaller than
	\(P_{(\ell)}\). Moreover, among the null \(p\)-values, exactly \(\ell-1\)
	values are smaller than \(P_{(\ell)}\). Therefore,
	\begin{equation*}
		R_\ell \le k+(\ell-1)+1=k+\ell.
	\end{equation*}
	
	If the null hypothesis corresponding to \(P_{(\ell)}\) is rejected and \(P_{(\ell)}\) has an overall rank \(R_\ell\), then by the
	definition of the GBS step-down procedure,
	\begin{equation*}
		P_{(\ell)}\le c_{R_\ell}.
	\end{equation*}
	Since the GBS critical constants \(\{c_i\}\) are non-decreasing in \(i\) and
	\(R_\ell\le k+\ell\), it follows that
	\begin{equation*}
		P_{(\ell)}
		\le c_{R_\ell}
		\le c_{k+\ell},
		\qquad \ell=1,\ldots,j.
	\end{equation*}
	
	Consequently, conditional on \(K_m=k\),
	\begin{equation*}
		\{V_m^{\mathrm{GBS}}\ge j\}
		\subseteq
		\left\{
		P_{(\ell)}\le c_{k+\ell},
		\ \ell=1,\ldots,j
		\right\},
	\end{equation*}
	and therefore
	\begin{equation*}
		\mathbb{P}\!\left(
		V_m^{\mathrm{GBS}}\ge j
		\,\middle|\,
		K_m=k
		\right)
		\le
		\mathbb{P}\!\left(
		P_{(\ell)}\le c_{k+\ell},
		\ \ell=1,\ldots,j
		\,\middle|\,
		K_m=k
		\right).
	\end{equation*}
	
	Since, conditional on \(K_m=k\), the null \(p\)-values are independent and
	identically distributed \(U(0,1)\) random variables, their order statistics
	satisfy
	\begin{equation*}
		(P_{(1)},\ldots,P_{(n)}\mid K_m=k)
		\stackrel{d}{=}
		(U_{(1)},\ldots,U_{(n)}),
	\end{equation*}
	where \(U_{(1)}\le\cdots\le U_{(n)}\) are the order statistics of \(n\)
	independent \(U(0,1)\) random variables. Hence
	\begin{equation*}
		\mathbb{P}\!\left(
		V_m^{\mathrm{GBS}}\ge j
		\,\middle|\,
		K_m=k
		\right)
		\le
		\mathbb{P}\!\left(
		U_{(\ell)}\le c_{k+\ell},
		\ \ell=1,\ldots,j
		\right).
	\end{equation*}
	
	Using the identity
	\begin{equation*}
		\mathbb{E}(Y)
		=
		\sum_{j\ge1}\mathbb{P}(Y\ge j)
	\end{equation*}
	for a nonnegative integer-valued random variable \(Y\), we obtain
	\begin{equation*}
		\mathbb{E}\!\left(V_m^{\mathrm{GBS}}\mid K_m=k\right)
		\le
		\sum_{j=1}^{n}
		\mathbb{P}\!\left(
		U_{(\ell)}\le c_{k+\ell},
		\ \ell=1,\ldots,j
		\right).
	\end{equation*}
	
	Now, since \(n=m-k\),
	\begin{align*}
		c_{k+\ell}
		&=
		\frac{(k+\ell)\alpha_m}
		{m+1-(k+\ell)(1-\alpha_m)}
		\\
		&=
		\frac{(k+\ell)\alpha_m}
		{n+k+1-(k+\ell)(1-\alpha_m)}
		\\
		&=
		b_{\ell,k}^{(n)}(\alpha_m).
	\end{align*}
	Hence Lemma~\ref{lem:gbs-shifted-null-bound} gives
	\begin{equation*}
		\mathbb{E}\!\left(V_m^{\mathrm{GBS}}\mid K_m=k\right)
		\le
		\frac{\alpha_m}{1-\alpha_m}(k+1).
	\end{equation*}
	
	Taking expectations with respect to \(K_m\) and using
	\(K_m\sim\mathrm{Bin}(m,p_m)\), we obtain
	\begin{align*}
		\mathbb{E}\!\left(V_m^{\mathrm{GBS}}\right)
		&\le
		\frac{\alpha_m}{1-\alpha_m}\mathbb{E}(K_m+1)
		\\
		&=
		\frac{\alpha_m}{1-\alpha_m}(mp_m+1).
	\end{align*}
	
	Since \(mp_m\to\infty\) and \(\alpha_m\to0\), we have
	\begin{equation*}
		mp_m+1=O(mp_m),
		\qquad
		\frac{1}{1-\alpha_m}=O(1).
	\end{equation*}
	Therefore,
	\begin{equation*}
		\mathbb{E}\!\left(V_m^{\mathrm{GBS}}\right)
		=
		O(\alpha_m mp_m).
	\end{equation*}
	
	This completes the proof of Corollary~\ref{cor:type1-bound}.
\end{proof}

\subsection*{Proof of Lemma~\ref{lem:deterministic-signal-crossing}}

\begin{proof}
	Under the alternative hypothesis \(H_{Ai}\),
	\begin{equation*}
		\frac{X_i}{\sigma}\sim N(0,1+u_m).
	\end{equation*}
	For \(0<t<1\), write
	\begin{equation*}
		z_t=\Phi^{-1}\left(1-\frac{t}{2}\right).
	\end{equation*}
	Since \(P_i\le t\) is equivalent to
	\(|X_i|/\sigma\ge z_t\), we have
	\begin{equation}
		\label{eq:Lem2_Gm_formula}
		G_m(t)
		=
		\mathbb{P}(P_i\le t\mid i\in\mathcal H_0^c)
		=
		2\bar\Phi\left(\frac{z_t}{\sqrt{1+u_m}}\right).
	\end{equation}
	
	We first record a monotonicity property that will be used to obtain the
	uniform bound. Put
	\begin{equation*}
		s_m=\sqrt{1+u_m}>1.
	\end{equation*}
	Since \(t=2\bar\Phi(z_t)\), it follows from
	\eqref{eq:Lem2_Gm_formula} that
	\begin{equation}
		\label{eq:Lem2_ratio_representation}
		\frac{G_m(t)}{t}
		=
		\frac{\bar\Phi(z_t/s_m)}
		{\bar\Phi(z_t)}.
	\end{equation}
	For \(z>0\), define
	\begin{equation*}
		R_m(z)
		=
		\frac{\bar\Phi(z/s_m)}
		{\bar\Phi(z)}.
	\end{equation*}
	Let
	\begin{equation*}
		h(x)=\frac{\phi(x)}{\bar\Phi(x)}
	\end{equation*}
	denote the standard normal hazard function. Since \(h\) is positive and
	increasing on \((0,\infty)\), we have
	\begin{align}
		\frac{d}{dz}\log R_m(z)
		&=
		h(z)-\frac{1}{s_m}h(z/s_m)
		\notag\\
		&=
		\{h(z)-h(z/s_m)\}
		+
		\left(1-\frac1{s_m}\right)h(z/s_m)
		>
		0.
		\label{eq:Lem2_ratio_derivative_positive}
	\end{align}
	Thus \(R_m(z)\) is increasing in \(z\). Since \(z_t\) is decreasing in \(t\),
	\eqref{eq:Lem2_ratio_representation} implies that \(G_m(t)/t\) is
	non-increasing in \(t\).
	
	Fix \(0<\varepsilon<q_*\), and put
	\begin{equation*}
		q_\varepsilon=q_*-\varepsilon.
	\end{equation*}
	Define
	\begin{equation}
		\label{eq:Lem2_tm_defn}
		t_m=q_\varepsilon\alpha_m p_m.
	\end{equation}
	We now establish the limiting behavior of \(G_m(t_m)\). Since
	\(q_\varepsilon\) is fixed and positive,
	\begin{equation}
		\label{eq:Lem2_tm_log_decomposition}
		2\log\left(\frac1{t_m}\right)
		=
		2\log\left(\frac1{\alpha_m p_m}\right)
		+
		2\log\left(\frac1{q_\varepsilon}\right).
	\end{equation}
	We next relate the calibration condition involving \(f_m/r_{\alpha_m}\) to
	the quantity \(\alpha_m p_m\). Since \(p_m\to0\) and
	\(\alpha_m\to0\) as \(m\to\infty\),
	\begin{equation}
		\label{eq:Lem2_calibration_equiv_log}
		\log\left(\frac{f_m}{r_{\alpha_m}}\right)
		=
		\log\left(\frac1{\alpha_m p_m}\right)+o(1).
	\end{equation}
Consequently, the first condition in
\eqref{eq:gbs-calibration-condition} implies
	\begin{equation}
		\label{eq:Lem2_calibration_alpha_p}
		\frac{2\log\{1/(\alpha_m p_m)\}}{u_m}
		\to C
		\quad \textrm{as } m\to\infty.
	\end{equation}
	Combining \eqref{eq:Lem2_tm_log_decomposition} and
	\eqref{eq:Lem2_calibration_alpha_p}, and using the fact that
	\(q_\varepsilon\) is fixed, we obtain
	\begin{equation}
		\label{eq:Lem2_tm_log_limit}
		\frac{2\log(1/t_m)}{u_m}
		\to C
		\quad \textrm{as } m\to\infty.
	\end{equation}
	By the standard Gaussian quantile expansion derived from Mills' ratio,
	\begin{equation}
		\label{eq:Lem2_tm_quantile_expansion}
		z_{t_m}^{2}
		=
		2\log\left(\frac1{t_m}\right)\{1+o(1)\}.
	\end{equation}
	Hence, by \eqref{eq:Lem2_tm_log_limit},
	\begin{equation}
		\label{eq:Lem2_tm_z_limit}
		\frac{z_{t_m}}{\sqrt{1+u_m}}
		\to
		\sqrt C
		\quad \textrm{as } m\to\infty.
	\end{equation}
	Therefore, by \eqref{eq:Lem2_Gm_formula},
	\begin{equation}
		\label{eq:Lem2_tm_Gm_limit}
		G_m(t_m)
		=
		2\bar\Phi\left(
		\frac{z_{t_m}}{\sqrt{1+u_m}}
		\right)
		\to
		2\bar\Phi(\sqrt C)
		=
		q_*
		\quad \textrm{as } m\to\infty.
	\end{equation}
	
	Now let \(0<q\le q_\varepsilon\). Then
	\(q\alpha_m p_m\le t_m\). Since \(G_m(t)/t\) is non-increasing in \(t\),
	we have
	\begin{equation}
		\label{eq:Lem2_ratio_monotonicity_application}
		\frac{G_m(q\alpha_m p_m)}
		{q\alpha_m p_m}
		\ge
		\frac{G_m(t_m)}{t_m}.
	\end{equation}
	Using \(t_m=q_\varepsilon\alpha_m p_m\), this gives
	\begin{equation}
		\label{eq:Lem2_q_lower_bound_ratio}
		\frac{G_m(q\alpha_m p_m)}{q}
		\ge
		\frac{G_m(t_m)}{q_\varepsilon}.
	\end{equation}
	By \eqref{eq:Lem2_tm_Gm_limit},
	\begin{equation}
		\label{eq:Lem2_ratio_limit}
		\frac{G_m(t_m)}{q_\varepsilon}
		\to
		\frac{q_*}{q_*-\varepsilon}
		=
		1+\frac{\varepsilon}{q_*-\varepsilon}
		\quad \textrm{as } m\to\infty.
	\end{equation}
	Therefore, for all sufficiently large \(m\),
	\begin{equation}
		\label{eq:Lem2_final_uniform_crossing}
		G_m(q\alpha_m p_m)
		\ge
		\left(
		1+\frac{\varepsilon}{2(q_*-\varepsilon)}
		\right)q,
	\end{equation}
	uniformly over \(0<q\le q_*-\varepsilon\).
	
	Taking
	\begin{equation*}
		\eta_\varepsilon
		=
		\frac{\varepsilon}{2(q_*-\varepsilon)}
	\end{equation*}
	completes the proof of Lemma~\ref{lem:deterministic-signal-crossing}.
\end{proof}

\subsection*{Proof of Lemma~\ref{lem:initial-crossing}}

\begin{proof}
	
	Let
	\begin{equation*}
		A_m
		=
		\log\left(\frac1{\alpha_m}\right),
		\qquad
		B_m
		=
		\log\left(\frac1{p_m}\right),
		\qquad
		D_m
		=
		\log(mp_m).
	\end{equation*}
	Since \(p_m\to0\), \(\alpha_m\to0\), and \(mp_m\to\infty\),
	\begin{equation*}
		A_m+B_m\to\infty,
		\qquad
		D_m\to\infty.
	\end{equation*}
	Moreover,
	\begin{equation}
		\label{eq:Lem3_log_identity}
		\log\left(\frac{m}{\alpha_m}\right)
		=
		A_m+B_m+D_m.
	\end{equation}
	
	Since
	\begin{equation*}
		c_1
		=
		\frac{\alpha_m}{m+\alpha_m}
		\sim
		\frac{\alpha_m}{m},
	\end{equation*}
	we have
	\begin{equation*}
		\log\left(\frac1{c_1}\right)
		=
		\log\left(\frac{m}{\alpha_m}\right)+o(1).
	\end{equation*}
	
	Let
	\begin{equation*}
		z_{c_1}
		=
		\Phi^{-1}
		\left(
		1-\frac{c_1}{2}
		\right).
	\end{equation*}
	
	By the standard Gaussian quantile expansion derived from Mills' ratio,
	\begin{equation}
		\label{eq:Lem3_quantile_expansion}
		z_{c_1}^{2}
		=
		2\log\left(\frac{m}{\alpha_m}\right)\{1+o(1)\}.
	\end{equation}
	
	Since \(p_m\to0\) and \(\alpha_m\to0\) as \(m\to\infty\),
	\begin{equation*}
		\log\left(\frac{f_m}{r_{\alpha_m}}\right)
		=
		\log\left(\frac1{\alpha_m p_m}\right)+o(1)
		=
		A_m+B_m+o(1).
	\end{equation*}
	
Combining this with the first condition in
\eqref{eq:gbs-calibration-condition}
	\begin{equation*}
		\frac{2\log(f_m/r_{\alpha_m})}{u_m}
		\to C
		\quad \textrm{as } m\to\infty,
	\end{equation*}
	we obtain
	\begin{equation}
		\label{eq:Lem3_um_expansion}
		u_m
		=
		\frac{2}{C}(A_m+B_m)\{1+o(1)\}.
	\end{equation}
	
	Put
	\begin{equation*}
		y_m
		=
		\frac{z_{c_1}}{\sqrt{1+u_m}}.
	\end{equation*}
	
	Using \eqref{eq:Lem3_log_identity},
	\eqref{eq:Lem3_quantile_expansion}, and
	\eqref{eq:Lem3_um_expansion}, we obtain
	\begin{align}
		\frac{y_m^{2}}{2}
		&=
		\frac{z_{c_1}^{2}}{2(1+u_m)}
		\notag\\
		&=
		\frac{C}{2}
		\frac{A_m+B_m+D_m}
		{A_m+B_m}
		\{1+o(1)\}
		\notag\\
		&=
		\frac{C}{2}
		\left\{
		1+
		\frac{D_m}{A_m+B_m}
		\right\}
		\{1+o(1)\}.
		\label{eq:Lem3_y_sq_bound}
	\end{align}
	
	Dividing \eqref{eq:Lem3_y_sq_bound} by \(D_m\), and using
	\(D_m\to\infty\) and \(A_m+B_m\to\infty\), we obtain
	\begin{equation}
		\label{eq:Lem3_y_sq_o_D}
		\frac{y_m^{2}}{2}
		=
		o(D_m).
	\end{equation}
	
	
	In particular,
	\begin{equation}
		\label{eq:Lem3_log_y_o_D}
		\log(1+y_m)
		=
		o(D_m).
	\end{equation}
	
	By Mills' lower bound, there exists a universal constant \(C_0>0\) such that
	\begin{equation}
		\label{eq:Lem3_Mills_lower}
		\bar\Phi(y)
		\ge
		C_0
		\frac{\exp(-y^{2}/2)}
		{1+y},
		\qquad y\ge0.
	\end{equation}
	
	Therefore,
	\begin{align}
		mp_mG_m(c_1)
		&=
		mp_m\,2\bar\Phi(y_m)
		\notag\\
		&\ge
		2C_0
		\exp(D_m)
		\frac{\exp(-y_m^{2}/2)}
		{1+y_m}.
		\label{eq:Lem3_mpG_lower}
	\end{align}
	
	Taking logarithms in \eqref{eq:Lem3_mpG_lower} and using
	\eqref{eq:Lem3_y_sq_o_D} and
	\eqref{eq:Lem3_log_y_o_D}, we obtain
	\begin{align}
		\log\{mp_mG_m(c_1)\}
		&\ge
		\log(2C_0)
		+
		D_m
		-
		\frac{y_m^{2}}{2}
		-
		\log(1+y_m)
		\notag\\
		&=
		D_m\{1+o(1)\}.
		\label{eq:Lem3_log_mpG_diverges}
	\end{align}
	
	Since \(D_m=\log(mp_m)\to\infty\), it follows from
	\eqref{eq:Lem3_log_mpG_diverges} that
	\begin{equation*}
		mp_mG_m(c_1)
		\to\infty
		\quad \textrm{as } m\to\infty.
	\end{equation*}
	
	This completes the proof of Lemma~\ref{lem:initial-crossing}.
	
\end{proof}

\subsection*{Proof of Lemma~\ref{lem:empirical-signal-crossing}}

\begin{proof}
	
	Let
	\begin{equation*}
		Q_{(1)}\le \cdots \le Q_{(K_m)}
	\end{equation*}
	denote the ordered signal \(p\)-values on outcomes where $K_m\ge1$, and let
	\begin{equation*}
		c_j
		=
		\frac{j\alpha_m}
		{m+1-j(1-\alpha_m)},
		\qquad j=1,\ldots,m,
	\end{equation*}
	denote the GBS critical constants.
	
	We first prove that
	\begin{equation}
		\label{eq:Lem4_R_crossing_goal}
		\mathbb{P}(R_m^{\mathrm{GBS}}\ge r_m)\to1
		\quad \textrm{as } m\to\infty.
	\end{equation}
	Define the event
	\begin{equation}
		\label{eq:Lem4_signal_order_crossing_goal}
		\mathcal E_m=\{K_m\ge r_m\}\cap
		\bigcap_{j=1}^{r_m}\{Q_{(j)}\le c_j\}.
	\end{equation}
	The order statistics in this event are evaluated only when
	\(K_m\ge r_m\), so all indices are defined. It suffices to show
	\(\mathbb P(\mathcal E_m)\to1\). Indeed, on \(\mathcal E_m\),
	\(P_{(j)}\le Q_{(j)}\le c_j\) for every \(1\le j\le r_m\),
	and the GBS step-down rule gives \(R_m^{\mathrm{GBS}}\ge r_m\).
	
	For \(t\in(0,1)\), define
	\begin{equation*}
		N_1(t)
		=
		\sum_{i\in\mathcal H_0^c}\mathbf{1}\{P_i\le t\}.
	\end{equation*}
	Conditional on \(K_m=k\),
	\begin{equation*}
		N_1(t)\sim \mathrm{Bin}(k,G_m(t)).
	\end{equation*}
	For \(j\le K_m\), the condition \(Q_{(j)}\le c_j\) is equivalent
	to \(N_1(c_j)\ge j\). The latter event is well-defined even when
	\(K_m<j\), and therefore
	\begin{equation*}
		\mathcal E_m=\bigcap_{j=1}^{r_m}\{N_1(c_j)\ge j\}.
	\end{equation*}
	Thus it suffices to show
	\begin{equation}
		\label{eq:Lem4_N1_crossing_goal}
		\mathbb{P}\left(N_1(c_j)\ge j,\ j=1,\ldots,r_m\right)\to1
		\quad \textrm{as } m\to\infty.
	\end{equation}
	
	Choose a sequence of positive integers \(\{L_m\}\) such that
	\(L_m\to\infty\) as \(m\to\infty\),
	\begin{equation*}
		L_m=o(mp_m),
		\qquad
		L_m=o(mp_mG_m(c_1)).
	\end{equation*}
	Such a choice is possible since \(mp_m\to\infty\) and, by
	Lemma~\ref{lem:initial-crossing}, \(mp_mG_m(c_1)\to\infty\).
	
	We first control the early ranks \(1\le j<L_m\). Since \(c_j\ge c_1\),
	\begin{equation*}
		N_1(c_j)\ge N_1(c_1),
		\qquad j\ge1.
	\end{equation*}
	Let
	\begin{equation*}
		\mathcal K_m^{(0)}
		=
		\left\{
		K_m\ge \frac12 mp_m
		\right\}.
	\end{equation*}
	Since \(K_m/(mp_m)\to1\) in probability,
	\begin{equation*}
		\mathbb{P}(\mathcal K_m^{(0)})\to1
		\quad \textrm{as } m\to\infty.
	\end{equation*}
	On \(\mathcal K_m^{(0)}\),
	\begin{equation*}
		\mathbb{E}\{N_1(c_1)\mid K_m\}
		=
		K_mG_m(c_1)
		\ge
		\frac12 mp_mG_m(c_1).
	\end{equation*}
	Since \(L_m=o(mp_mG_m(c_1))\), the Chernoff bound gives
	\begin{equation*}
		\mathbb{P}\{N_1(c_1)<L_m\mid K_m\}\to0
		\quad \textrm{as } m\to\infty,
	\end{equation*}
	uniformly on \(\mathcal K_m^{(0)}\). Hence
	\begin{equation}
		\label{eq:Lem4_early_c1_bound}
		\mathbb{P}\{N_1(c_1)<L_m\}\to0
		\quad \textrm{as } m\to\infty.
	\end{equation}
	Consequently,
	\begin{equation}
		\label{eq:Lem4_early_rank_conclusion}
		\mathbb{P}\left(N_1(c_j)\ge j,\ j=1,\ldots,L_m-1\right)\to1
		\quad \textrm{as } m\to\infty.
	\end{equation}
	
	We now consider the bulk ranks \(L_m\le j\le r_m\). Write
	\begin{equation*}
		q_j=\frac{j}{mp_m}.
	\end{equation*}
	Then
	\begin{equation*}
		0<q_j\le q_*-\varepsilon.
	\end{equation*}
	For these \(j\)'s, since \(L_m\to\infty\), we have
	\(j(1-\alpha_m)>1\) for all sufficiently large \(m\). Hence
	\begin{equation*}
		m+1-j(1-\alpha_m)\le m,
	\end{equation*}
	and therefore
	\begin{equation}
		\label{eq:Lem4_cj_lower_bound}
		c_j
		=
		\frac{j\alpha_m}
		{m+1-j(1-\alpha_m)}
		\ge
		\frac{j\alpha_m}{m}
		=
		q_j\alpha_m p_m.
	\end{equation}
	Since \(G_m\) is nondecreasing and \(c_j\ge q_j\alpha_m p_m\) by
	\eqref{eq:Lem4_cj_lower_bound}, Lemma~\ref{lem:deterministic-signal-crossing}
	implies that, for some \(\eta_\varepsilon>0\),
	\begin{equation}
		\label{eq:Lem4_Gm_bulk_lower_bound}
		G_m(c_j)
		\ge
		G_m(q_j\alpha_m p_m)
		\ge
		(1+\eta_\varepsilon)q_j
	\end{equation}
	uniformly over \(L_m\le j\le r_m\).
	
	Let
	\begin{equation*}
		\mathcal K_m
		=
		\left\{
		K_m\ge (1-\eta_\varepsilon/4)mp_m
		\right\}.
	\end{equation*}
	Since \(K_m/(mp_m)\to1\) in probability,
	\begin{equation*}
		\mathbb{P}(\mathcal K_m)\to1
		\quad \textrm{as } m\to\infty.
	\end{equation*}
	On \(\mathcal K_m\), by \eqref{eq:Lem4_Gm_bulk_lower_bound},
	\begin{align*}
		K_mG_m(c_j)
		&\ge
		(1-\eta_\varepsilon/4)mp_m(1+\eta_\varepsilon)q_j \\
		&=
		(1-\eta_\varepsilon/4)(1+\eta_\varepsilon)j.
	\end{align*}
	Since
	\[
	(1-\eta_\varepsilon/4)(1+\eta_\varepsilon)
	=
	1+\frac{3\eta_\varepsilon}{4}-\frac{\eta_\varepsilon^2}{4},
	\]
	we may decrease \(\eta_\varepsilon\), if necessary, so that
	\begin{equation*}
		(1-\eta_\varepsilon/4)(1+\eta_\varepsilon)
		\ge
		1+\eta_\varepsilon/2.
	\end{equation*}
	Thus
	\begin{equation}
		\label{eq:Lem4_mean_margin}
		K_mG_m(c_j)
		\ge
		(1+\eta_\varepsilon/2)j,
	\end{equation}
	uniformly over \(L_m\le j\le r_m\) on \(\mathcal K_m\).
	
	Therefore, by the Chernoff bound, conditional on \(K_m\),
	\begin{equation}
		\label{eq:Lem4_chernoff_bulk}
		\mathbb{P}\left(N_1(c_j)<j\mid K_m\right)
		\le
		\exp(-a_\varepsilon j)
	\end{equation}
	for some constant \(a_\varepsilon>0\), uniformly over \(L_m\le j\le r_m\)
	on \(\mathcal K_m\).
	
	Since \(\mathbb{P}(\mathcal K_m^c)=o(1)\), the union bound yields, for all
	sufficiently large \(m\),
	\begin{align}
		&\mathbb{P}\left(
		N_1(c_j)<j
		\textrm{ for some }L_m\le j\le r_m
		\right)
		\notag\\
		&\le
		\mathbb{P}\left(
		N_1(c_j)<j
		\textrm{ for some }L_m\le j\le r_m,\,
		\mathcal K_m
		\right)
		+
		\mathbb{P}(\mathcal K_m^c)
		\notag\\
		&\le
		\sum_{j=L_m}^{r_m}
		\exp(-a_\varepsilon j)
		+
		o(1).
		\label{eq:Lem4_bulk_failure_union}
	\end{align}
	Since \(L_m\to\infty\),
	\begin{equation*}
		\sum_{j=L_m}^{r_m} \exp(-a_\varepsilon j)
		\le
		\sum_{j=L_m}^{\infty} \exp(-a_\varepsilon j)
		=
		\frac{\exp(-a_\varepsilon L_m)}
		{1-\exp(-a_\varepsilon)}
		\to0
		\quad \textrm{as } m\to\infty.
	\end{equation*}
	Hence, by \eqref{eq:Lem4_bulk_failure_union},
	\begin{equation}
		\label{eq:Lem4_bulk_rank_conclusion}
		\mathbb{P}\left(
		N_1(c_j)<j
		\textrm{ for some }L_m\le j\le r_m
		\right)
		\to0
		\quad \textrm{as } m\to\infty.
	\end{equation}
	
	Combining \eqref{eq:Lem4_early_rank_conclusion} and
	\eqref{eq:Lem4_bulk_rank_conclusion}, we obtain
	\begin{equation}
		\label{eq:Lem4_N1_crossing_conclusion}
		\mathbb{P}\left(N_1(c_j)\ge j,\ j=1,\ldots,r_m\right)\to1
		\quad \textrm{as } m\to\infty.
	\end{equation}
	By the count-event identity above, this proves
	\(\mathbb P(\mathcal E_m)\to1\). Since
	\(\mathcal E_m\subseteq\{R_m^{\mathrm{GBS}}\ge r_m\}\), we obtain
\begin{equation}
		\label{eq:Lem4_R_crossing_conclusion}
		\mathbb{P}(R_m^{\mathrm{GBS}}\ge r_m)\to1
		\quad \textrm{as } m\to\infty.
	\end{equation}
	
	We now translate this crossing result into a Type-II error bound. Let
	\[
	A_m=\{R_m^{\mathrm{GBS}}\ge r_m\}.
	\]
	Write \(S_m^{\mathrm{GBS}}\) for the number of rejected signal hypotheses. Since
	\[
	S_m^{\mathrm{GBS}}
	=
	R_m^{\mathrm{GBS}}-V_m^{\mathrm{GBS}},
	\]
	on \(A_m\) we have
	\[
	S_m^{\mathrm{GBS}}
	\ge
	r_m-V_m^{\mathrm{GBS}}.
	\]
	Therefore, on \(A_m\),
	\begin{align*}
		T_m^{\mathrm{GBS}}
		&=
		K_m-S_m^{\mathrm{GBS}} \\
		&\le
		K_m-r_m+V_m^{\mathrm{GBS}}.
	\end{align*}
	On \(A_m^c\), the trivial bound \(T_m^{\mathrm{GBS}}\le K_m\) holds. Hence
	\begin{equation}
		\label{eq:Lem4_T_pointwise_bound}
		T_m^{\mathrm{GBS}}
		\le
		K_m-r_m+V_m^{\mathrm{GBS}}
		+
		r_m\mathbf{1}\{R_m^{\mathrm{GBS}}<r_m\}.
	\end{equation}
	Indeed, on \(A_m\) this is the preceding bound, while on \(A_m^c\) the
	right-hand side is at least \(K_m\).
	
	Taking expectations gives
	\begin{equation}
		\label{eq:Lem4_T_expectation_bound}
		\mathbb{E}(T_m^{\mathrm{GBS}})
		\le
		\mathbb{E}(K_m)-r_m
		+
		\mathbb{E}(V_m^{\mathrm{GBS}})
		+
		r_m\mathbb{P}(R_m^{\mathrm{GBS}}<r_m).
	\end{equation}
	
	Now \(\mathbb{E}(K_m)=mp_m\), and
	\begin{equation*}
		r_m=(q_*-\varepsilon)mp_m+O(1).
	\end{equation*}
	By Corollary~\ref{cor:type1-bound},
	\begin{equation*}
		\mathbb{E}(V_m^{\mathrm{GBS}})
		=
		O(\alpha_mmp_m)
		=
		o(mp_m),
	\end{equation*}
	because \(\alpha_m\to0\). Moreover, by
	\eqref{eq:Lem4_R_crossing_conclusion},
	\[
	\mathbb{P}(R_m^{\mathrm{GBS}}<r_m)\to0,
	\]
	and since \(r_m=O(mp_m)\),
	\begin{equation}
		\label{eq:Lem4_remainder_goal}
		r_m\mathbb{P}(R_m^{\mathrm{GBS}}<r_m)
		=
		o(mp_m).
	\end{equation}
	
	Consequently, by \eqref{eq:Lem4_T_expectation_bound},
	\begin{align*}
		\mathbb{E}(T_m^{\mathrm{GBS}})
		&\le
		mp_m-(q_*-\varepsilon)mp_m+o(mp_m) \\
		&=
		mp_m(1-q_*+\varepsilon)+o(mp_m).
	\end{align*}
	This concludes the proof of Lemma~\ref{lem:empirical-signal-crossing}.
	
\end{proof}

\bibliographystyle{apalike}
\bibliography{GBS_ABOS_References}

\end{document}